\documentclass[reqno,12pt]{amsart}
\usepackage[utf8]{inputenc}
\usepackage{amsmath,amssymb,amsthm,mathrsfs,color,times,textcomp,verbatim,yfonts}
\usepackage{subcaption}
\usepackage[normalem]{ulem}
\usepackage[export]{adjustbox}
\usepackage{esint}
\usepackage{xcolor}
\usepackage{array}
\usepackage[colorlinks=true]{hyperref}
\hypersetup{urlcolor=blue, citecolor=red, linkcolor=blue}

\usepackage[square,numbers]{natbib}

\def \r{\mathfrak{r}}
\def \C{\mathcal{C}}

\numberwithin{equation}{section}
\theoremstyle{plain}
\newtheorem{theorem}{Theorem}[section]
\newtheorem{proposition}[theorem]{Proposition}
\newtheorem{lemma}[theorem]{Lemma}

\newtheorem{claim}{Claim}

\usepackage{graphicx}

\theoremstyle{definition}

\newtheorem{remark}[theorem]{Remark}

\title[Ancient solution to $\alpha$-CSF]{ancient finite entropy flows by powers of curvature in $\mathbb{R}^2$}
\author{Kyeongsu Choi}
\address{School of Mathematics, Korea Institute for Advanced Study, 85 Hoegiro, Dongdaemun-gu, Seoul 02455, Republic of Korea.}
\email{choiks@kias.re.kr}
\author{Liming Sun}
\address{Department of Mathematics, University of British Columbia, Vancouver, BC, V6T 1Z2, CA.}
\email{lsun@math.ubc.ca}
\date{\today}
\subjclass[2010]{Primary 53C44, 53A04; Secondary 35K55}
\keywords{}

\begin{document}

\maketitle

\begin{abstract}
    We show the existence of non-homothetic ancient flows by powers of curvature embedded in $\mathbb{R}^2$ whose entropy  is finite.  We determine the Morse indices and kernels of the linearized operator of shrinkers to the flows, and construct ancient flows by using unstable eigenfunctions of the linearized operator. 
\end{abstract}
  
\section{Introduction}

Given $\alpha>0$, the $\alpha$-curve shortening flow ($\alpha$-CSF) is a family of complete convex curves $\Gamma_t $ embedded in $\mathbb{R}^2$ which evolves by the $\alpha$-power-of-curvature. Namely,  the position vector $\mathbf{X}(\cdot,t)$ of $\Gamma_t$ satisfies 
\begin{align}\label{eq:main-flow}
    \frac{\partial \mathbf{X}}{\partial t}(p,t)=\kappa^{\alpha}(p,t)\mathbf{N}(p,t),
\end{align}
where $\kappa$ is the curvature and $\mathbf{N}$ is {inward pointing} unit normal vector of $\Gamma_t$.

\bigskip

We say that a flow $\Gamma_t$ is \textbf{ancient} if it exists for $t\in (-\infty,T)$ for some $T\in \mathbb{R}\cup \{+\infty\}$.  Geometric flows satisfy parabolic equations so that there are in general only a few number of ancient flows. For example, \citet{Wang11} showed that a closed convex embedded ancient curve shortening flow (CSF)\footnote{Curve shortening flow means the $\alpha$-CSF with $\alpha=1$.} sweeping the entire plane is a shrinking circle, and Daskalopoulos-Hamilton-Sesum \cite{daskalopoulos2010} showed that a closed convex embedded  ancient CSF is a shrinking circle or an Angenent oval.\footnote{It looks like a shortening paper clip sweeping a slab.} See also \citet*{bourni2019convex} for the classification of non-compact  ones.

Ancient flows have been intensively studied in the mean curvature flow, a higher dimensional version of the CSF. In particular, ancient mean curvature flows are useful to investigate singularities. See \cite{angenent2019unique,angenent2018uniqueness,brendle2018uniqueness,brendle2019uniqueness,choi2018ancient,choi2019ancient-white,chodosh2020mean} (c.f. Ricci flow  \cite{brendle2018ancient,angenent2019unique-B,brendle2020uniqueness}).

The $\alpha$-CSF is a fully nonlinear flow, which behaves like the $\alpha$-Gauss curvature flow in many aspects. In particular, if $\alpha=\frac{1}{3}$ ({$\alpha=\frac{1}{n+2}$ in higher dimensions) then a $\alpha$-CSF remains a $\alpha$-CSF under any affine transform (of determinant one) of the ambient space. The affine normal flow ($\frac{1}{3}$-CSF) have been widely studied due to its beauty from affine geometry. For example, \citet{chen2015classifying} showed that an ancient closed convex affine normal flow must be a shrinking ellipse (see an alternate proof by \citet{ivaki2016classification}). See also \cite{loftin2008ancient} for higher dimensions.
%\bigskip

\bigskip

\citet*{andrews2016flow} introduced an important notion of entropy for $\alpha$-CSF. We recall that the support function  $u_{z_0}$ with respect to $z_0\in \mathbb{R}^2$ is
\begin{align*}
    u_{z_0}(\theta):=\max_{z\in \Omega}\langle (\cos \theta,\sin\theta),z-z_0\rangle,
\end{align*}
and the entropy $\mathcal{E}_\alpha (\Omega)$ of a bounded convex region $\Omega\subset\mathbb{R}^2$ and its boundary $\partial \Omega$ is defined by 
\begin{equation}
\mathcal{E}_\alpha (\partial\Omega)=\mathcal{E}_\alpha (\Omega)=\sup_{z_0\in \Omega}\mathcal{E}_\alpha(\Omega,z_0),
\end{equation}
where $\mathcal{E}_\alpha(\Omega,z_0)$ is
\begin{align}
    \mathcal{E}_\alpha(\Omega,z_0)=\begin{cases}
    \frac{\alpha}{\alpha-1}\log\left( \fint_{\mathbb{S}^1}u_{z_0}^{1-\frac{1}{\alpha}}(\theta)d\theta\right)- \frac{1}{2}\log\frac{|\Omega|}{\pi}&\text{if }\alpha\neq 1,\\ \fint_{\mathbb{S}^1}\log u_{z_0}(\theta)d\theta- \frac{1}{2}\log\frac{|\Omega|}{\pi}&\text{if }\alpha=1.
    \end{cases}
\end{align}
Here $|\Omega|$ denotes the area of it.

In \cite{andrews2016flow}, they showed that the entropy $\mathcal{E}_\alpha (\Gamma_t)$ of the $\alpha$-CSF decreases with respect to $t$. Hence, we say that an ancient $\alpha$-CSF has \textbf{finite entropy} if
\begin{equation}
\lim_{t\to -\infty}\mathcal{E}_\alpha (\Gamma_t)< +\infty.
\end{equation}
Clearly, self-shrinking ancient solutions has finite entropy, since the entropy does not change under homothetic transformation. However, every non-homothetic ancient $\alpha$-CSF discovered in previous researches including \cite{angenent1992doughnuts} and \cite{bourni2020ancient} do not have finite entropy. See also \cite{choi2020uniqueness-choi} for a higher dimensional analogue. Indeed, the entropy of every non-homothetic ancient $\alpha$-CSF with $\alpha\in (\frac{2}{3},1]$ must diverge by \cite{daskalopoulos2010} and \cite{bourni2020ancient}.  In this paper, we present families of non-homothetic closed ancient $\alpha$-CSFs which converge to a self-shrinker\footnote{If $\Gamma_t=(- t)^{\frac{1}{\alpha+1}} \Gamma_{-1}$ is the $\alpha$-CSF, then we call $\Gamma_{-1}$ a self-shrinker or a shrinker.} as $t\to -\infty$ after rescaling. Then, their entropy is less than that of the limiting shrinker, namely the ancient flows have the finite entropy. See Theorem \ref{thm:exist-type1}.

\bigskip

To construct ancient flows asymptotic to a self-shrinking ancient flow, we first recall the classification result of self-shrinkers.
\begin{theorem}[Andrews \cite{andrews2003classification}]\label{prop:class-ben}
If $\alpha\in[\frac{1}{8},+\infty)\backslash \{1/3\}$, then the shrinker of \eqref{eq:main-flow} is a circle (denote it as $\Gamma_\alpha^c$). If $\alpha=\frac{1}{3}$, then  a shrinker is an ellipse. If $\alpha \in (0,\frac{1}{8})$, then a shrinker is a circle or a curve ${\Gamma_{\alpha}^k}$ with $k$-fold symmetry, where $3\leq  k \in \mathbb{N}$ with $k<\sqrt{1+1/\alpha}$. The curves ${\Gamma_{\alpha}^k}$ depend smoothly on $\alpha<\frac{1}{k^2-1}$ and converge to regular $k$-sided poloygons as $\alpha\searrow 0$ and to circles as $\alpha\nearrow \frac{1}{k^2-1}$.  See Table \ref{table1} and Figure \ref{fig:converge} for illustrations.
\end{theorem}
   \begin{table}[ht]
   \begin{tabular}{ >{\centering\arraybackslash}m{1in} | >{\arraybackslash}m{2in} }
   \hline
   $\alpha$ &  $\Gamma_{\alpha}^c\text{ and }{\Gamma_{\alpha}^k}$\\
    \hline
    $[\frac{1}{8},+\infty)\backslash\frac13 $& \includegraphics[width=0.11\textwidth]{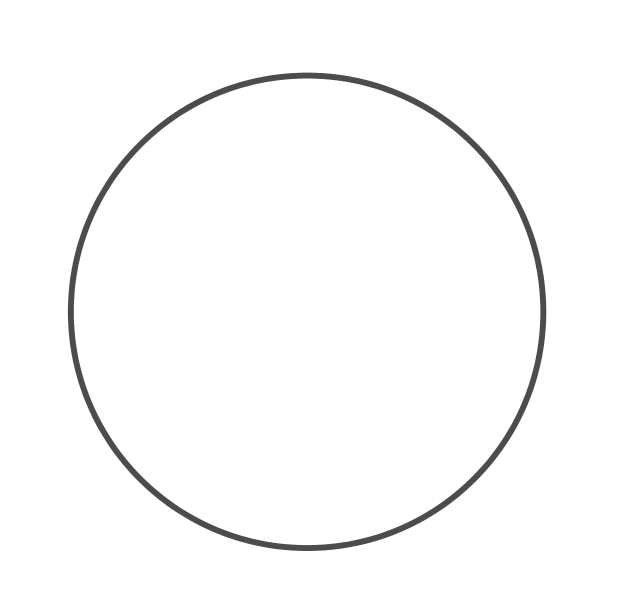}   \\ \hline
    $[\frac{1}{15},\frac{1}{8})$ & \includegraphics[width=0.11\textwidth]{shrinker-1.png}\includegraphics[width=0.1\textwidth]{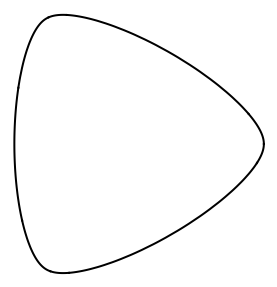} \\ \hline
    $[\frac{1}{24},\frac{1}{15})$ & \includegraphics[width=0.11\textwidth]{shrinker-1.png}\includegraphics[width=0.1\textwidth]{shirnker-3.png}\includegraphics[width=0.1\textwidth]{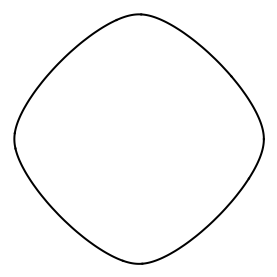}  \\ \hline
    $\cdots$& $\cdots$\\
    \hline
  \end{tabular}
  \vspace{0.5cm}
  \caption{Enumeration of shrinkers for different $\alpha$.}
  \label{table1}
  \end{table}
\begin{figure}[ht]
    \centering
    \includegraphics[width=0.32\linewidth]{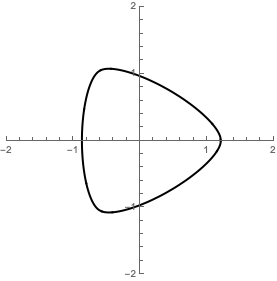}
    \includegraphics[width=0.32\linewidth]{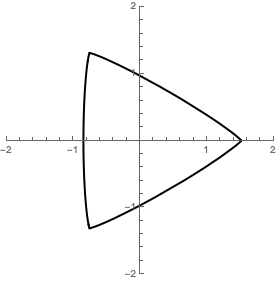}
    \includegraphics[width=0.32\linewidth]{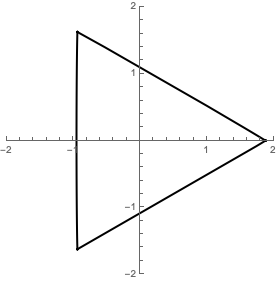}
    \caption{The shape of ${\Gamma_{\alpha}^k}$ (normalized by \eqref{eq:h2}) when $k=3$,  $\alpha=\frac19,\frac{1}{16},\frac{1}{100}$ from left to right.}
    \label{fig:converge}
\end{figure}

To fix the asymptotic self-shrinking ancient flow, we consider the normalized flow $\bar \Gamma_\tau$ defined by
\begin{equation}\label{eq:normalized_flow}
\bar{\boldsymbol{X}}(p,\tau)=(1+\alpha)^{-\frac{1}{\alpha+1}}e^{\tau}\boldsymbol{X}(p,-e^{-(1+\alpha)\tau}),
\end{equation}
By Proposition \ref{prop:u_evol}  the support function $\bar u(\theta,\tau)$ of $\bar{\boldsymbol{X}}$ with respect to the origin satisfies
  \begin{align}\label{eq:hat-u-tau}
  \bar u_\tau=-(\bar u_{\theta\theta}+\bar u)^\alpha+\bar u.
  \end{align}
Hence, the support function $h$ of a self-shrinker $\Gamma$ with respect to the origin satisfies
\begin{align}\label{eq:h2}
   h_{\theta\theta}+h=h^{-1/\alpha},
   \end{align}
and thus the difference $v=\bar u - h$ satisfies
\begin{align}\label{intro:eq:v-flow}
v_\tau=-{(h_{\theta\theta}+h+v_{\theta\theta}+v)^{-\alpha}}+(h+v):= \mathcal{L}_{\Gamma}(v)+ E_{\Gamma}(v).
\end{align}
Here $\mathcal{L}_\Gamma$ is the linearization of the above equation at $v=0$
 \begin{align}\label{eq:L-def}
 \mathcal{L}_{\Gamma}(v):={\alpha h^{1+\frac{1}{\alpha}}}(v_{\theta\theta}+v)+v
 \end{align}
and  
\begin{equation}
|E_{\Gamma}(v)|\leq C|v_{\theta\theta}+v|^2,
\end{equation}
for small enough $v_{\theta\theta}+v$. See Proposition \ref{prop:error_expansion} for details.
\bigskip
 
 It is easy to see that the Jacobi operator $\mathcal{L}_{\Gamma}$ is a self-adjoint operator on the space $L_h^2(\mathbb{S}^1)=\{f:\int_{\mathbb{S}^1}f^2h^{-1-1/\alpha}<\infty\}$, and thus it has a sequence of eigenvalues and eigenfunctions  which form the basis of $L_h^2(\mathbb{S}^1)$. We are able to characterize its kernel and  \textbf{Morse index}\footnote{The dimension of negative space of $-\mathcal{L}$.} as follows.
\begin{theorem}[cf. Proposition \ref{prop:eig} and Theorem \ref{thm:main-eig}]\label{intro:thm:sep-L} Suppose $0<\alpha \neq \frac13$. 
\begin{enumerate}
\item The Morse index of $\mathcal{L}_{\bar\Gamma_{\alpha}^k}$ is $2k-1$, and  $\ker \mathcal{L}_{\bar\Gamma_{\alpha}^k}=span\{h_\theta\}$, where $h$ is the support function of $\bar\Gamma_{\alpha}^k$.
\item The Morse index of $ \mathcal{L}_{\bar\Gamma_{\alpha}^c}$ is $2\lceil\sqrt{1+1/\alpha}\rceil-1$.\footnote{$\lceil x\rceil$ denotes least integer greater than or equal to $x$.}  If $\alpha=\frac{1}{k^2-1}$, then $\ker \mathcal{L}_{\bar\Gamma_{\alpha}^c}=span\{\cos k\theta,\sin k\theta\}$. Otherwise $\ker \mathcal{L}_{\bar\Gamma_{\alpha}^c}=\emptyset$.
\end{enumerate}
\end{theorem}

The center manifold theory in functional analysis provides the existence of an $I$-parameter family of ancient solutions to a class of fully nonlinear parabolic equations, where $I$ is the Morse index. See \citet[Chapter 9]{lunardi2012analytic}. However, using the contraction mapping method, we can show the existence of such ancient solutions and even including sharp asymptotic behaviors of the solutions with layer structures. See \citet{choi2019ancient} and Caffarelli-Hardt-Simon \cite{caffarelli1984} for quasilinear parabolic and elliptic PDEs. Here comes the second main theorem of our paper.

\begin{theorem}[cf. Theorem \ref{thm:exist-type1}]\label{intro:thm:exist1}
Let $\alpha \neq \frac{1}{3}$ and $\lambda_1\leq \cdots \leq \lambda_I <0$ denote the negative eigenvalues of $\mathcal{L}_\Gamma$ where $\Gamma=\bar\Gamma_{\alpha}^k$ or $\bar\Gamma_{\alpha}^c$ and $I$ is the Morse index.  There exists $\beta\in (0,1)$, $\varepsilon_0>0$ and an injective continuous map $\mathcal{S}:B_{\varepsilon_0}(0)(\subset\mathbb{R}^{I-3})\to C^{2,\beta}(\mathbb{S}^1\times (-\infty,-1])$ such that for each $\boldsymbol{a}=(a_1,\cdots,a_{I-3})\in \mathbb{R}^{I-3}$ the image $v=\mathcal{S}(\boldsymbol{a})$ is an ancient solution to \eqref{intro:eq:v-flow}. Moreover, if \,$3<k\leq I$ and $\boldsymbol{a},\boldsymbol{b}\in \mathbb{R}^{I-3}$ satisfy $a_{k-3}-b_{k-3} \neq 0$ and  $a_j-b_j = 0$ for all $j>k-3$, then $\mathcal{S}$ satisfies
\begin{equation}\label{intro:eq:layer1}
 \mathcal{S}(\boldsymbol{a})(\theta,\tau)-\mathcal{S}(\boldsymbol{b})(\theta,\tau)=(a_{k-3}-b_{k-3}) e^{-\lambda_{k}\tau}\varphi_k(\theta) +o(e^{-\lambda_{k}\tau})
\end{equation}
when $\lambda_{k-1}<\lambda_k$, and
\begin{multline}\label{intro:eq:layer2}
 \mathcal{S}(\boldsymbol{a})(\theta,\tau)-\mathcal{S}(\boldsymbol{b})(\theta,\tau)\\=e^{-\lambda_{k}\tau}\sum_{i=k-1}^{k}(a_{i-3}-b_{i-3}) \varphi_i(\theta) +o(e^{-\lambda_{k}\tau})
\end{multline}
when $\lambda_{k-1}=\lambda_k$, where $\varphi_i$ are eigenfunctions of $\mathcal{L}_\Gamma$ with the eigenvalue $\lambda_i$ and $\langle \varphi_i,\varphi_j\rangle_{L^2_h}=\delta_{ij}$. In particular, $\mathcal{S}(\boldsymbol{0})(\theta,\tau)=h(\theta)$ corresponds to the shrinker. 
\end{theorem}

\begin{remark}
Notice that the first three eigenfunctions of $\mathcal{L}_{\Gamma}$ are $h,\cos\theta,\sin\theta$ by Proposition \ref{prop:eig}, which accounts for dilations and transitions of the non-rescaled $\alpha$-CSF. See Proposition \ref{prop:rescaling_center}. Therefore, we consider $(I-3)$-parameter family of ancient solutions rather than $I$-parameter.

Moreover, if $\Gamma=\Gamma_\alpha^c$, then rotations accounts for $1$-parameter. Namely, Theorem \ref{intro:thm:exist1} provides a $(I-4)$-parameter family ancient flows converging to a round shrinking circle up to rigid motions and dilations. 

In short, given $\frac{1}{k^2-1}\leq \alpha < \frac{1}{(k-1)^2-1}$ with $3\leq k \in \mathbb{N}$, by Theorem \ref{intro:thm:exist1} there exist, up to rigid motions and dilations, a $(2k-5)$-parameter family of closed convex ancient $\alpha$-CSFs converging to a round shrinking circle and a $(2m-3)$-parameter family of closed convex ancient $\alpha$-CSFs converging to a shrinking $m$-fold symmetric curve for each integer $3\leq m<k$. 
\end{remark}

In an following paper, the authors will classify ancient finite entropy $\alpha$-CSFs, and show that the solutions in Theorem \ref{intro:thm:exist1} are the all solutions up to transitions and dilations with exhibiting the layer structure \eqref{intro:eq:layer1}-\eqref{intro:eq:layer2}.
  
An outline of our paper is in order. In Section 2, we devote to studying the spectrum of the linear operator $\mathcal{L}$. In Section 3, we construct ancient solutions converge with finite entropy by contraction mapping theorem. 
   \bigskip

\noindent \textbf{Acknowledgements.} The authors are grateful to Christos Mantoulidis for fruitful discussion, and also thankful to Shibing Chen, Beomjun Choi, John Loftin, and Mohammad N. Ivaki for their comments and suggestions. K. Choi is supported by KIAS Individual Grant MG078901.
   
   \section{Spectra of  linearized operators}\label{sec:spectrum}
We begin by deriving the evolution equation of the support function $\bar u$ of the normalized flow $\bar \Gamma_\tau$ given by \eqref{eq:normalized_flow}.
\begin{proposition}\label{prop:u_evol}
Let $\bar \Gamma_\tau$ be a normalized $\alpha$-CSF satisfying \eqref{eq:normalized_flow}. The support function $\bar u$ of  $\bar \Gamma_\tau$ satisfies
  \begin{align}
  \bar u_\tau=-\bar \kappa^\alpha+\bar u=-(\bar u_{\theta\theta}+\bar u)^{-\alpha}+\bar u.
  \end{align}
\end{proposition}   
   
   \begin{proof}
   By using \eqref{eq:normalized_flow}, we have
   \begin{align}
   \bar \kappa=(1+\alpha)^{\frac{1}{1+\alpha}}e^{-\tau}\kappa,
   \end{align}
   and thus
\begin{align}
 \partial_\tau \bar{\boldsymbol{X}}=(1+\alpha)^{\frac{\alpha}{1+\alpha}}e^{-\alpha \tau}\partial_tX+\bar X= {\bar \kappa^\alpha}\mathbf{N}+\bar {\boldsymbol{X}}.
\end{align}   
Therefore, $\bar u=\langle \bar{\boldsymbol{X}},\mathbf{N}\rangle$ and $\bar u_{\theta\theta}+\bar u=\bar \kappa^{-1}$ yield the desired evolution equation.
   \end{proof}

We are interested in normalized ancient flows $\bar \Gamma_\tau$ converging to a shrinker $\Gamma=\bar \Gamma_\alpha^k$ or $\bar \Gamma_\alpha^c$ as $\tau \to -\infty$. Namely, the difference $v=\bar u-h$ converges to zero, where $h$ is the suppose function of $\Gamma$ satisfying \eqref{eq:h2}. Moreover, the evolution equation \eqref{intro:eq:v-flow} has the linearized operator $\mathcal{L}$ given by \eqref{eq:L-def}.
\begin{equation}
\mathcal{L}_\Gamma=\alpha h^{1+\frac{1}{\alpha}} (\partial_\theta^2 +1)+1.
\end{equation}
here $h$ is the support function of $\Gamma$. We shall abbreviate $\mathcal{L}_\Gamma$ as $\mathcal{L}$ whenever there is no confusion.
   
  We introduce  the space  $L^2_h(\mathbb{S}^1)=L^2(\mathbb{S}^1, h^{-1-1/\alpha}d\theta)$ with norm $||f||_h^2=\int_{\mathbb{S}^1}f^2h^{-1-1/\alpha}$. It is equipped with the inner product
   \begin{align}\label{eq:h-inner}
   (f,g)_h=\int_{\mathbb{S}^1}fgh^{-1-\frac{1}{\alpha}}d\theta.
   \end{align}
   Since $h>0$ on $\mathbb{S}^1$ and \eqref{eq:h2},  this norm is equivalent to the standard $L^2$ norm.

   It is easy to see that $\mathcal{L}$ is a self-adjoint operator on $L^2_{h}$. Since $\mathcal{L}$ is an elliptic operator on a compact space, thus $-\mathcal{L}$ has a sequence of eigenvalues $\lambda_1\leq \lambda_2\leq\cdots$. We remind that an eigenfunction $\varphi\in L_h^2(\mathbb{S}^1)$ and the corresponding eigenvalue $\lambda \in \mathbb{R}$ satisfy
   \begin{align}\label{eq:L-eig-h}
   {\alpha h^{1+\frac{1}{\alpha}}}(\varphi_{\theta\theta}+\varphi)+(\lambda+1)\varphi=0,\quad \text{ on }\quad \mathbb{S}^1.
   \end{align}
Moreover, there exists a sequence of  the pairs  $(\lambda_i,\varphi_i)$ of eigenvalues and eigenfunctions such that $\lambda_i\leq \lambda_{i+1}$, $\lim_{i\to \infty} \lambda_i= +\infty$, $(\varphi_i,\varphi_j)_h=\delta_{ij}$, and $\text{span}\{\varphi_1,\varphi_2,\cdots\}=L_h^2(\mathbb{S}^1)$.  
  
    In this section, we will study eigenfunctions with negative or zero eigenvalues of $-\mathcal{L}$.     
  
   \begin{proposition}\label{prop:eig}
   There are some known eigenvalues for $-\mathcal{L}$.
   \begin{enumerate}
   \item $\lambda=-1-\alpha$ is an eigenvalue with the eigenfunction $\varphi=h$. Since $h$ is always positive, $\lambda=-1-\alpha$ is the lowest  eigenvalue.
   \item $\lambda=-1$ is an eigenvalue with the eigenfunctions $\varphi=\sin\theta,\cos\theta$.
   \item $\mathcal{L}_{\bar \Gamma_{\alpha}^c}=\alpha (\partial^2_{\theta}+1)+1$ has eigenvalues 
   \begin{align}\label{eq:eig-h=1}
   \lambda_1=-\alpha-1,\quad \lambda_{2l}=\lambda_{2l+1}=\alpha(l^2-1)-1,\quad l\geq 1
   \end{align}
   with the eigenfunctions $\cos(l\theta)$ and $\sin(l\theta)$. Notice that $-\mathcal{L}$ has an eigenvalue $\lambda=0$ only when $\alpha=1/({l^2-1})$ for some $l\geq 2$. 
   
   \item $\mathcal{L}_{\bar \Gamma_{\alpha}^k}$ has zero eigenvalue $\lambda=0$ with eigenfunction $\varphi= h_\theta$. More importantly, $\lambda=0$ is simple.
    \end{enumerate}
   \end{proposition}
   \begin{proof}
   (1), (2), (3) are easy to verify. For (4), it is obtained by  differentiating \eqref{eq:h2} with respect to  $\theta$, which gives $h_{\theta}$ satisfies \eqref{eq:L-eig-h} when $\lambda=0$. Indeed, $h_\theta$ arises from rotations of ${\bar \Gamma_{\alpha}^k}$. \citet[Lemma 7.3]{andrews2003classification} shows that the eigenspace of $\lambda=0$ has dimension ONE, which is $span\{h_\theta\}$.  
   \end{proof}
   
In Proposition \ref{prop:eig}, we characterize all eigenfunctions of $-\mathcal{L}_{\bar \Gamma_{\alpha}^c}$ and  neutral eigenfunctions of $-\mathcal{L}_{\bar \Gamma_{\alpha}^k}$. Thus, we will focus on $\Gamma=\bar \Gamma_{\alpha}^k$ and consider negative eigenvalues of $-\mathcal{L}_{\bar \Gamma_{\alpha}^k}$. We shall simply write $\mathcal{L}=\mathcal{L}_{\bar \Gamma_{\alpha}^k}$ for the rest of this section.
   
   The following lemma is equivalent to \cite[Lemma 5]{andrews2002non} whose proof needs Brunn-Minkowski inequality there. We give a direct proof here.
  \begin{lemma}\label{lem:eig-NO}There is NO eigenvalue of $-\mathcal{L}$ in $(-1-\alpha,-1)$.
  \end{lemma}
  \begin{proof}Suppose $\varphi$ is an eigenfunction of $-\mathcal{L}$ satisfying $(\varphi,h)_h=0$ and \eqref{eq:L-eig-h}. Then there exists $c$ such that $\tilde \varphi=\varphi-ch$ satisfy $\int_{\mathbb{S}^1}\tilde \varphi=0$. Then $\int_{\mathbb{S}^1}\tilde \varphi^2-|\tilde \varphi_\theta|^2\leq 0$. Multiplying \eqref{eq:L-eig-h} by $\tilde \varphi\,h^{-1-\frac{1}{\alpha}}$ and integrating  over $\mathbb{S}^1$ give
  \begin{align}\label{eq:eig-2}
  \alpha\int_{\mathbb{S}^1}(\varphi_{\theta\theta}+\varphi)\tilde \varphi+(\lambda+1)\int_{\mathbb{S}^1}\varphi\tilde \varphi\,h^{-1-\frac{1}{\alpha}}=0.
  \end{align}
  Let us simplify the left-hand side. First, using the fact $\int_{\mathbb{S}^1}\varphi\,h^{-\frac{1}{\alpha}}=0$, we have 
  \[\int_{\mathbb{S}^1}\varphi\tilde \varphi h^{-1-\frac{1}{\alpha}}=\int_{\mathbb{S}^1}\varphi^2h^{-1-\frac{1}{\alpha}}=(\varphi,\varphi)_h> 0.\]
  Second,
  \begin{align*}
  \int_{\mathbb{S}^1}(\varphi_{\theta\theta}+\varphi)\tilde \varphi=&\int_{\mathbb{S}^1}[\tilde \varphi_{\theta\theta}+\tilde \varphi]\tilde \varphi+c\int_{\mathbb{S}^1}[h_{\theta\theta}+h]\tilde \varphi\\
  =&\int_{\mathbb{S}^1}(\tilde \varphi^2-\tilde \varphi_\theta^2)+c\int_{\mathbb{S}^1}h^{-\frac{1}{\alpha}}(\varphi-ch)\leq -c^2\int_{\mathbb{S}^1}h^{-1-\frac{1}{\alpha}}\leq 0,
  \end{align*}
  where in the second equality we used \eqref{eq:h2}.

  If $\lambda\in(-1-\alpha,-1)$ is an eigenvalue of $-\mathcal{L}$, inserting the above two inequalities into the LHS of \eqref{eq:eig-2}, one could find out the LHS $<0$. Contradiction.
  \end{proof}

It follows from Proposition \ref{prop:eig} and Lemma \ref{lem:eig-NO} that $-\mathcal{L}$ has eigenvalues
\begin{align}\label{eq:L-eig-enum}
    \lambda_1<\lambda_2=\lambda_3< \lambda_4\leq\cdots
\end{align}
where $\lambda_1=-1-\alpha$, $\lambda_2=\lambda_3=-1$.

\begin{theorem}\label{thm:main-eig} Suppose $k\geq 3$, $\alpha\in (0,{1}/({k^2-1}))$. The negative eigenspace of $-\mathcal{L}_{\bar\Gamma_\alpha^k}$ has dimension $2k-1$. In particular, 
\begin{enumerate}
\item If $k$ is odd, every negative eigenvalue except $\lambda_1$ has the eigenspace of dimension two. If $k$ is even, every negative eigenvalue except $\lambda_1,\lambda_k,\lambda_{k+1}$ has the eigenspace of dimension two. In both cases, any eigenfunction of $\lambda_{2l}$ and $\lambda_{2l+1}$, $1\leq l\leq k-1$, have $2l$ zeros. 
\item Furthermore, $\lambda_{2k}=0$ and $\lambda_{2k+1}>0$ are simple, namely
\[\lambda_{2k-1}<\lambda_{2k}(=0)<\lambda_{2k+1}<\lambda_{2k+2}\leq \cdots.\]
In addition, both $\varphi_{2k}$ and $\varphi_{2k+1}$ have $2k$ nodal sets.
\end{enumerate}
\end{theorem}

Easily one can see the dimension of eigenspace of each eigenvalue is at most $2$. This is because \eqref{eq:L-eig-h} is a second order ODE and it has at most two linearly independent solutions.

For a function $\varphi$,  the term \textit{zeros} (or \textit{nodal sets}) refers to the set $\{\theta:\varphi(\theta)=0\}$. The term \textit{nodal domain} refers to the connected components of the complement of the nodal sets.

\begin{lemma}\label{lem:nodal-nondeg}
Each eigenfunction $\varphi$ satisfies $\varphi^2+|\varphi'|^2>0$. Any eigenfunctions $\varphi$, except $\varphi\in span\{h\}$, has even number of zeros and even number of nodal domains.
\end{lemma}

\begin{proof}
If $\varphi(\theta_0)=\varphi'(\theta_0)=0$ at some point $\theta_0 \in \mathbb{S}^1$, then we have $\varphi=0$ on $\mathbb{S}^1$ by the uniqueness of solution to the second order ODE. Hence, $\varphi^2+|\varphi'|^2>0$ everywhere. Therefore, if $\varphi(\theta_0)=0$ at some $\theta_0\in \mathbb{S}^1 $ then $\varphi(\theta_0+\epsilon)\varphi(\theta_0-\epsilon)<0$ for small enough $\epsilon$. Namely, $\varphi$ change its signs at zeros. Hence, $\varphi$ has even number of zeros and thus it has even number of nodal domains. 
\end{proof}

\begin{lemma}\label{lem:eig-two}
Suppose that $\lambda$ has a two dimensional eigenspace.  Then, its eigenfunctions have the same number of nodal sets.
\end{lemma}
\begin{proof}
This follows from the Sturm separation theorem. 
For reader's convenience, we give a proof. 
Suppose $span\{\varphi,\psi\}$ are the eigenspace of $\lambda$, where $\varphi,\psi$ are linearly independent. Then the Wronskian of $\varphi$ and $\psi$ is not zero for any $\theta$, that is 
\begin{align}\label{eq:wronskian}
W[\varphi,\psi](\theta)=\begin{vmatrix} \varphi(\theta) & \psi(\theta) \\ \varphi'(\theta) & \psi'(\theta)\end{vmatrix}\neq 0.
\end{align}
If $\psi$ and $\varphi$ has different number of nodal sets, by Lemma \ref{lem:nodal-nondeg}, then without loss of generality one nodal set $\varphi$ is strictly contained in one nodal set of $\psi$. That is, there exists $\theta_1$ and $\theta_2$ such that 
\[\varphi(\theta_1)=\varphi(\theta_2)=0,\quad \varphi'(\theta_1)\varphi'(\theta_2)<0,\quad \psi(\theta_1)\psi(\theta_2)>0.\]
However, we will get
\[W(\theta_1)W(\theta_2)=\varphi'(\theta_1)\psi(\theta_1)\varphi'(\theta_2)\psi(\theta_2)<0.\]
This is not possible, because $W$ does not change sign.
\end{proof}
Because the above lemma, we are eligible to say the number of nodal sets corresponding to an eigenvalue $\lambda$.

\begin{lemma}\label{lem:nodal-two}
Suppose an eigenvalue $\lambda_i$ has a two dimensional eigenspace, then for any $\lambda_j>\lambda_i$, the number of nodal sets of the eigenfunctions corresponding to $\lambda_j$ is greater than that of $\lambda_i$. Similarly, if $\lambda_j<\lambda_i$, then the number of nodal sets corresponding to $\lambda_j$ is less than that of $\lambda_i$.
\end{lemma}
\begin{proof}
Suppose the the eigenspace of $\lambda_i$ is $span\{\psi_i,\tilde \psi_i\}$ and take any eigenfunction $\varphi_j$ corresponding to $\lambda_j$. Assume $\varphi_j(\theta_0)=0$ for some $\theta_0$,  One can find $\omega\in [0,2\pi]$ such that $(\cos \omega) \psi_{i}(\theta_0)+(\sin\omega)\tilde \psi_{i}(\theta_0)=0$.   Notice $\varphi_i=(\cos \omega) \psi_{i}+(\sin\omega) \tilde\psi_{i}$ is an eigenfunction of $\lambda_{i}$. Since $\varphi_i(\theta_0)=\varphi_j(\theta_0)=0$,  the conclusions follow from the Sturm-Picone comparison theorem. For reader's convenience, we sketch the proof of the case $\lambda_j>\lambda_i$. The other case is similar. It suffices to prove that there is at least a zero of $\varphi_j$ which lies strictly between any two consecutive zeros of $\varphi_i$.  Assume that $\varphi_i$ has two consecutive zeros $a$ and $b$,  and $\varphi_j$ has no zero in $(a,b)$. Without loss of generality, we assume $\varphi_i>0$ and $\varphi_j>0$ in $(a,b)$. 

Denote the Wronskian $W(\theta)=\varphi_i\varphi_j'-\varphi_i'\varphi_j$, then we can directly calculate $W'=(\lambda_i-\lambda_j)\frac{1}{\alpha }h^{-1-1/\alpha}\varphi_j\varphi_i$. We have the following Picone's identity
\begin{align}\label{eq:picone}
\left(\frac{\varphi_i}{\varphi_j}W\right)'=\frac{1}{\alpha }h^{-1-\frac{1}{\alpha}}(\lambda_i-\lambda_j)\varphi_i^2-\left(\frac{W}{\varphi_j}\right)^2
\end{align}
wherever $\varphi_j\neq 0$.

If $\varphi_j(a)>0$ and $\varphi_j(b)>0$, then we integrate \eqref{eq:picone} from $a$ to $b$. 
\[\frac{\varphi_i}{\varphi_j}W\Big|_{a}^{b}<0.\]
This contradicts to $\varphi_i(a)=\varphi_i(b)=0$.

If $\varphi_j(a)=0$ and $\varphi_j(b)>0$, then $\varphi_j'(a)>0$. Integrating \eqref{eq:picone} from $a-\epsilon$ to $b$ for $\epsilon>0$ small enough, we obtain 
\[\frac{\varphi_i}{\varphi_j}W\Big|_{a-\epsilon}^{b}<0\]
However, we know $\varphi_i(a-\epsilon)>0$, $\varphi_j(a-\epsilon)>0$ and $W(a-\epsilon)<0$ for sufficiently small $\epsilon>0$. Namely, the above quantity is positive. Hence it is an obvious contradiction.

 The case of $\varphi_j(a)>0$ with $\varphi_j(b)=0$ and $\varphi_j(a)=\varphi_j(b)=0$ can be ruled out in the same manner.
\end{proof}
Recall that the support function $h$ of $\bar \Gamma_k^\alpha$ is $k$-fold symmetric. It has  $2k$ critical points. By rotating $\bar \Gamma_k^\alpha$, we may assume 
\begin{equation}\label{eq:h-symmetry}
h'(n\pi/k)=0,
\end{equation}
for all $n\in \mathbb{Z}$. Then $h$ has even reflection symmetry with respect to $n\pi/k$ for any $n\in \mathbb{Z}$, namely $h(\theta)=h(2n\pi/k-\theta)$ for any $n\in \mathbb{Z}$.

 \begin{lemma}\label{lem:eig-one}
 Suppose that the eigenspace of $\lambda_i\neq \lambda_1$ is $span\{\varphi_i\}$, namely $\lambda_i$ is simple.  Then, either $\varphi_i$ has at least $2k$ zeros, or $\varphi_i$ has exactly $k$ zeros and $k$ must be even. In the second case, zeros of $\varphi_i$ are $\{2n\pi/k:n\in\mathbb{Z}\}$ or $\{(2n+1)\pi/k:n\in\mathbb{Z}\}$ modulo $2\pi$, where $h$ satisfies \ref{eq:h-symmetry}.
\end{lemma}
\begin{proof}
By the symmetry of $h$,  $\varphi_i(2n\pi/k-\theta)$ is an eigenfunction of $\lambda_i$. Since the eigenspace of $\lambda_i$ has dimension one, we must have $\varphi_i(\theta)=c\varphi_i(2n\pi/k-\theta)$ for some $c(n)\neq 0$ and any $\theta\in \mathbb{S}^1$. Replacing $\theta$ by $2n\pi/k-\theta$, one gets $\varphi_i(2n\pi/k-\theta)=c(n)\varphi_i(\theta)$.  Thus, $c(n)$ must be $1$ or $-1$.
\bigskip

If $c(n)=1$, then $\varphi_i$ has even reflection symmetry with respect to $n\pi/k$. Then, we have $\varphi_i'(n\pi/k)=0$, and because Lemma \ref{lem:nodal-nondeg}, we also $\varphi_i(n\pi/k)\neq 0$ for such $n$. If $c(n)=-1$, then $\varphi_i$ has an odd reflection symmetry with respect to $n\pi/k$. Obviously, we have $\varphi_i(n\pi/k)=0$ for such $n$. Conversely, if $\varphi_i(n\pi/k)\neq 0$ then $\varphi_i$ has even reflection symmetry with respect to $n\pi/k$. If $\varphi_i(n\pi/k)=0$, then $\varphi_i$ has odd reflection symmetry with respect to $n\pi/k$.

\bigskip

Now, we consider $\varphi_i$ on $[0,n\pi/k]$. We divide it into three cases.

 First, if $\varphi_i$ has a zero in $ (0,n\pi/k)$, then the reflect symmetries of $\varphi$ guarantees at least $2k$ zeros in $\mathbb{S}^1$. 
 
 Second, if $\varphi_i(0)=\varphi_i(n\pi/k)=0$ and has no zero inside $ (0,n\pi/k)$, then after the reflection symmetries of $\varphi$ guarantees at least $2k$ zeros in $\mathbb{S}^1$.
 
 Last, $\varphi_i$ has only one zero on the endpoint of $[0,n\pi/k]$, say $\varphi_i(0)=0$ and $\varphi_i\neq 0 $ for $(0,n\pi/k]$. In this case, the previous paragraph shows that $\varphi_i$ has even reflection symmetry with respect to $(2n+1)\pi/k$ and has odd reflection symmetry with respect to $2n\pi/k$ for any $n$. Moreover, $k$ must be even since $\varphi_i$ has even number of nodal sets. Counting the zeros of $\varphi_i$, we find it is $k$ in this case. 
\end{proof}

Now we can prove the first part of Theorem \ref{thm:main-eig}.

\begin{proof}[Proof of Theorem \ref{thm:main-eig} Part (1)]
We will use the induction to prove that $\lambda_{2l}=\lambda_{2l+1}$ for any $1\leq l<k$, except that $\lambda_{k}\leq \lambda_{k+1}$ when $k$ is even. In any case, eigenfunctions corresponds to $\lambda_{2l}$ and $\lambda_{2l+1}$ have $2l$ nodal domains.
If $l=1$, then $\lambda_2=\lambda_3$ by Proposition \ref{prop:eig}. Any eigenfunction in $span\{\cos\theta,\sin\theta\}$ has 2 nodal domains.
Suppose the induction is complete for any $l$ such that $2l+1\leq k-1$, that is 
\[\lambda_1<\lambda_2=\lambda_3<\cdots<\lambda_{2l}=\lambda_{2l+1}<0\]  
and those eigenfunctions of $\lambda_{2j}=\lambda_{2j+1}$ have $2j$ nodal domains for $j\leq l$. 

\bigskip
It follows from Courant nodal domain theorem \cite[VI.6]{courant1953methods}, $\varphi_{2(l+1)}$ has at most $2(l+1)$ nodal domains. %In fact since the nodal domains has to be even number, then $\varphi_{2(l+1)}$ has at most $2(l+1)$ nodal domain. 
Because $2(l+1)\leq 2(k-1)$, Lemma \ref{lem:eig-one} implies that $\lambda_{2(l+1)}$ will be repeated unless $k$ is even and $2(l+1)=k$. %However, since we assume $k$ is odd in this case, $\lambda_{2(l+1)}$ has to be repeated.

\bigskip
Let us first consider the case that $\lambda_{2(l+1)}$ is repeated. We can continue the induction. The dimension of the eigenspace of each eigenvalue is at most 2, thus $\lambda_{2l+3}=\lambda_{2(l+1)}>\lambda_{2l+1}=\lambda_{2l}$.  Now we only need to prove $\varphi_{2l+2}$ and $\varphi_{2l+3}$ has $2l+2$ nodal domains. First, they have the same number of nodal sets by Lemma \ref{lem:eig-two}, while Lemma \ref{lem:nodal-two} implies that they must  have at least $2l+2$ nodal sets. Combining the previous upper bound on the number of nodal sets, our induction for $l+1$ is complete. 

\bigskip
If $k$ is even and $\lambda_{2(l+1)}=\lambda_k$ is simple, then Lemma \ref{lem:eig-one} says $\varphi_k$ has $k$ zeros.  Now, we consider $\lambda_{k+1}$. By the Courant nodal domain theorem and Lemma \ref{lem:nodal-nondeg}, any eigenfunction associated to $\lambda_{k+1}$ also has $k$ zeros. Therefore, the second part of Lemma \ref{lem:nodal-two} implies that $\lambda_{k+1}$ is also simple.
%\begin{claim}\label{clm:simple}
%If $\lambda_k$ is simple, then $\lambda_{k+1}$ is also simple
%\end{claim}
%We will put off the proof of this claim a little bit. If this claim is true, 
Then, we have $\lambda_{2l+2}<\lambda_{2l+3}<\lambda_{2l+4}$. By Lemma \ref{lem:eig-one}, $\lambda_{2l+4}$ is repeated, and its eigenfunctions has $2l+4=k+2$ zeros. The induction on $l+1$ and $l+2$ is complete. Now, since $2(l+2)>k$ in this case, Lemma \ref{lem:eig-one} shows the rest negative eigenvalue are all repeated, therefore it can be continued as the previous case.
%\[\lambda_{2l+2}<\lambda_{2l+3}\]
\bigskip

Putting everything together, we have the following relations
\[
\lambda_1<\lambda_2=\lambda_3<\cdots<\lambda_{2l}= \lambda_{2l+1}<\cdots<\lambda_{2k-2}=\lambda_{2k-1}\]
except that when $k$ is even, the relation $\lambda_k=\lambda_{k+1}$ will be replaced by $\lambda_k\leq \lambda_{k+1}$.
All these eigenvalues are negative, because $\lambda=0$ has the eigenfunction $h_\theta$ with $2k$ nodal sets and  $\varphi_{2k-2}$ has $2k-2$ nodal sets. So, Lemma \ref{lem:nodal-two} says $\lambda_{2k-2}<0$.
\end{proof}

Next, in order to show Part (2) of Theorem \ref{thm:main-eig}, we consider the following eigenvalue problems in the Hilbert space $H^1([0,\pi/k])$ with various boundary conditions
\begin{align}
\psi''+\psi=-\frac{1}{\alpha }h^{-1-\frac{1}{\alpha}}(\mu+1)\psi,\quad \psi(0)=\psi(\pi/k)=0\tag{DD},\label{eq:DD-eig}\\
\psi''+\psi=-\frac{1}{\alpha }h^{-1-\frac{1}{\alpha}}(\mu+1)\psi,\quad \psi(0)=\psi'(\pi/k)=0\tag{DN},\label{eq:DN-eig}\\
\psi''+\psi=-\frac{1}{\alpha }h^{-1-\frac{1}{\alpha}}(\mu+1)\psi,\quad \psi'(0)=\psi(\pi/k)=0\tag{ND},\label{eq:ND-eig}\\
\psi''+\psi=-\frac{1}{\alpha }h^{-1-\frac{1}{\alpha}}(\mu+1)\psi,\quad \psi'(0)=\psi'(\pi/k)=0\tag{NN}.\label{eq:NN-eig}
\end{align}
Here we also assume $h$ satisfies \eqref{eq:h-symmetry}. According to the Sturm-Liouville theory, the eigenvalues $\mu_i^{AB}$ for the problems $(AB)$ where $A,B=D$ or $N$ satisfy,
\[\mu_1^{AB}<\mu_2^{AB}<\mu_3^{AB}<\cdots\]
and the eigenfunction $\psi_i^{AB}$ corresponding to $\mu_i^{AB}$ have $i-1$ zeros.

For \eqref{eq:NN-eig}, it is easy to know $h\in H^1([0,\pi/k])$ and it is an eigenfunction to the first eigenvalue $\mu_1^{NN}=-1-\alpha$. 
\begin{proposition}\label{prop:neu-eig2} For $0<\alpha<{1}/({k^2-1})$, we have $\mu_2^{NN}>0$ of \eqref{eq:NN-eig}.
\end{proposition}
\begin{proof}
We shall write $\mu_2=\mu_2^{NN}$ for short within this proposition.

 First, $\mu_2$ can not be equal to zero. Otherwise, we will have an eigenfunction $\psi_2$ on $[0,\pi/k]$ such that
\[\psi_2''+\psi_2=-\frac{1}{\alpha }h^{-1-\frac{1}{\alpha}}\psi_2,\quad \psi_2'(0)=\psi_2'(\pi/k)=0.\]
By reflecting $\psi_2$ about $n\pi/k$ evenly for any $n$, we can extend $\psi_2$ to a smooth function defined on $\mathbb{S}^1$. This contradicts to the fact that the eigenspace of $\mathcal{L}$ for $\lambda=0$ is one dimensional, because $\psi_2\not\in span\{h_\theta\}$.

%\textcolor{blue}{after reflection, one actually know that the function is just $C^2$.}

Second, suppose $\mu_2<0$ and $\psi_2$ is an eigenfunction. In the following Lemma \ref{lem:eta}, we get a function $\eta(\theta)$ with
\[\eta''+\eta=-\frac{1}{\alpha }h^{-1-\frac{1}{\alpha}}\eta\]
with $\eta(0)>0$ and $\eta'(0)=0$, $\eta(\pi/k)<0$, $\eta'(\pi/k)<0$.  See Figure \ref{fig:eta} for illustration.

We claim that $\eta$ has only one zero in $(0,\pi/k)$. 

In fact, on the contrary assume $\eta(\theta_0)=\eta(\theta_1)=0$ for $0<\theta_0<\theta_1<\pi/k$. One can define a new function $\tilde \eta(\theta)$ such that it equals $\eta(\theta)$ if $\theta\in[\theta_0,\theta_1]$ and zero elsewhere. Then obviously $\tilde \eta\in H^1([0,\pi/k])$ and
\begin{align}\label{eq:tilde-eta} 
\int_0^{\pi/k} \tilde\eta_\theta^2-\tilde \eta^2=\int_0^{\pi/k}\frac{1}{\alpha }h^{-1-1/\alpha} \tilde\eta^2
\end{align}
Recall the following variational characterization of $\mu_1^{DD}$, 
\[\mu_1^{DD}=\inf_{u\in H^1([0,\pi/k]), u\not\equiv 0}\left\{\left.\frac{\int_0^{\pi/k} u_\theta^2-u^2}{\int_0^{\pi/k} u^2\frac{1}{\alpha }h^{-1-1/\alpha}} -1\right|u(0)=u(\pi/k)=0\right\}.\]
The infimum is achieved by the first eigenfunction of \eqref{eq:DD-eig}. Using (3) in Proposition \ref{prop:eig}, we have 
$\mu_1^{DD}=0$ and eigenfunction $\psi_1^{DD}= h_\theta$, because $h_\theta$ does not change sign in $(0,\pi/k)$. 
However, \eqref{eq:tilde-eta} implies $\mu_1^{DD}<0$. Contradiction. The claim is proved. 

%Otherwise we will contradict the fact that the lowest Dirichlet eigenvalue of \eqref{eq:NN-eig} is 0 and $h_\theta$ is the eigenfunction.}  %See Figure \ref{fig:eta} for illustration.

Let $a$ be the only zero of $\eta$ in $(0,\pi/k)$ and $b$ be that of $\psi_2$. Without loss of generality, we assume $\psi_2(\theta)>0$ when $\theta\in(0,b)$.  Otherwise one can work on $-\psi_2$.
Define $W[\psi_2,\eta]=\psi_2\eta'-\psi_2'\eta$. Then, we have $W(0)=0$, $W(\pi/k)>0$ and 
\[W'(\theta)=\frac{1}{\alpha }h^{-1-\frac{1}{\alpha}}\mu_2\psi_2\eta.\]

If $a\geq b$, then $W(b)\geq 0$ while $W'<0$ in $(0,b)$. This is impossible because of $W(0)=0$. 

If $a<b$, then $W(b)<0$, $W'<0$ on $(b,\pi/k)$. This contradicts to $W(\pi/k)>0$. Therefore $\mu_2$ can not be negative.
\end{proof}

\begin{figure}[ht]
\includegraphics[width=0.5\linewidth]{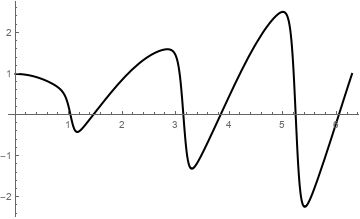}
\caption{The graph of $\eta$ with $\eta(0)=1$ when $\alpha=1/16$ and $k=3$.}
\label{fig:eta}
\end{figure}

\begin{lemma}\label{lem:eta}
Let $k\geq 3$ and $0<\alpha<1/(k^2-1)$. There exists a smooth function $\eta$ on $[0,2\pi]$ satisfying 
\begin{align}\label{eq:eta-ode}
\eta''+\eta+\frac{1}{\alpha}h^{-1-1/\alpha}\eta=0
\end{align}
and $\eta(0)>0$, $\eta'(0)=0$, $\eta(\pi/k)<0$ and $\eta'(\pi/k)<0$.
\end{lemma}
\begin{proof}
We will use some notations in \cite[Lemma 7.2]{andrews2003classification}. 
Consider the function $U(\alpha, r, \theta)$ defined by
\begin{align*}
U_{\theta \theta}+U &=U^{-\frac{1}{\alpha}} \\
U_{\theta}(\alpha, r, 0) =0,\quad&  U_{\theta}(\alpha, r, \Theta(\alpha, r)) =0 \\
U_{\theta}(\alpha, r, \theta) &<0, \quad 0<\theta<\Theta(\alpha, r) \\
U(\alpha, r, 0) &=r U(\alpha, r, \Theta(\alpha, r))
\end{align*}
where $\Theta=\Theta(\alpha,r)$ is the period function defined in \cite[Definition 2.1]{andrews2003classification}. Moreover, one can find 
\begin{align}
U(\alpha,r,\Theta(\alpha,r))&=\left(\frac{2\alpha}{1-\alpha}\cdot\frac{1-r^{1-1/\alpha}}{r^2-1}\right)^{\frac{\alpha}{\alpha+1}},\\
 U(\alpha,r, 0)&=rU(\alpha,r,\Theta(\alpha,r)).\label{eq:h-}
\end{align}

We will omit dependence on $\alpha$ of $U$ and $\Theta$ in what follows. It follows from \cite{andrews2003classification} that for each $\alpha\in(0,1/(k^2-1))$, there exists a unique $r^*\geq 1$ such that $\Theta(r_*)=\pi/k$. The support function $h$ is given by $h (\theta)=U(r_*,\theta)$. Define $\eta(\theta)=\frac{d}{dr}U(r,\theta)|_{r=r_*}$. Then obviously  $\eta$ satisfies \eqref{eq:eta-ode}. Since $U_\theta(r,0)=0$ for any $r$, we have $\eta_\theta(0)=0$.  Note that \eqref{eq:h-} implies 
\[U(r,0)=r\left(\frac{2\alpha}{1-\alpha}\cdot\frac{1-r^{1-1/\alpha}}{r^2-1}\right)^{\frac{\alpha}{\alpha+1}}\]
Differentiating  with respect to $r$ implies $\eta(0)>0$.

%It has been shown in \cite[Lemma 7.2]{andrews2003classification} there that $\eta$ satisfy the differential equation and $\eta(0)>0$ and $\eta'(0)=0$. Here we only need to show the fact on $\theta=\pi/k$. 

Since $U_\theta(r,\Theta(r))=0$, Differentiating with respect to $r$ gives
\begin{align}\label{ap:eta-bdry}
\eta_{\theta}(\Theta(r))+U_{\theta\theta}(r,\Theta(r))\frac{d}{dr}\Theta(r)=0.
\end{align}
Note that $d/dr\,\Theta(r)>0$ if $\alpha\in (0,1/3)$ by \cite{andrews2003classification}. Also $U_{\theta\theta}(r_*,\pi/k)\geq 0$, because $U(r_*,\theta)$ attains the minimum at $r=r_*$.  Here $U_{\theta\theta}(r_*,\pi/k)$ can not be 0, otherwise combined with $U_{\theta}(r_*,\pi/k)=0$, one gets $U$ is a constant. Inserting $r=r_*$ to the above equation,  one can see $\eta_\theta(\pi/k)<0$.

 On the other hand, it follows from \eqref{eq:h-} that 
 \[U(r,\Theta(r))=\left(\frac{2\alpha}{1-\alpha}\cdot\frac{1-r^{1-1/\alpha}}{r^2-1}\right)^{\frac{\alpha}{\alpha+1}}\]
 Taking the derivative with respect to $r$ of the above equation reveals $\frac{d}{dr}U(r,\Theta(r))<0$. Therefore $\eta(\pi/k)=\frac{d}{dr} U(r,\Theta(r))|_{r=r_*}<0$.
\end{proof}

\begin{proof}[Proof of Theorem \ref{thm:main-eig} Part (2)]
Since $\lambda_{2k-1}$ has a two dimensional eigenspace with $2k-2$ nodal sets, Lemma \ref{lem:nodal-two} implies $\lambda_{2k-1}<\lambda_{2k}$. The Courant nodal domain and Lemma \ref{lem:nodal-nondeg} imply that the eigenfunctions assosicated to $\lambda_{2k}$ must have $2k$ nodal sets. We need to show $\lambda_{2k}=0$. 

Towards a contradiction, suppose that $\lambda_{2k}<0$. $h_\theta$ is an eigenfunction corresponding to the eigenvalue $0$ and it also has $2k$ nodal sets. Thus, Lemma \ref{lem:nodal-two} implies that $\lambda_{2k}$ must be simple. Therefore, Lemma \ref{lem:eig-one} says that the eigenfunction $\varphi_{2k}$ is even-reflection-symmetric with respect to $n\pi/k$ for all $n\in \mathbb{N}$ and it has exactly $2k$ nodal sets. Thus, the restriction of $\varphi_{2k}$ on $[0,\pi/k]$ is a Neumann eigenfunction to \eqref{eq:NN-eig}. Since $\varphi_{2k}$ changes its sign exactly once in $[0,\pi/k]$, we have $\lambda_{2k}\geq \mu_2^{NN}$ which is the second Neumann eigenvalue. However, Proposition \ref{prop:neu-eig2} says $\mu_2^{NN}>0$.

 Since $\lambda_{2k-1}<0$ has $2k-2$ nodal domains and $0$ is an eigenvalue having $2k$ nodal domains, we have $\lambda_{2k}= 0$. $\lambda_{2k}$ is simple by Proposition \ref{prop:eig}, and thus we will have the next eigenvalue $\lambda_{2k+1}>0$. The nodal sets of eigenfunction associated to $\lambda_{2k+1}$ is $2k$ by Courant nodal domain theorem. Then the second part of Lemma \ref{lem:nodal-two} implies that $\lambda_{2k+1}$ also have to be simple.
%Next, we only need to show there is no other eigenvalue between $\mu_2<\lambda<0$, then the part (2) is established. Suppose such $\lambda$ exists and has an eigenfunction $\varphi$, then Lemma \ref{lem:eig-one} implies $\varphi$ has to be reflection symmetric with respect to $n\pi/k$ for any $n$ and it has exactly $2k$ nodal sets. Thus the restriction of $\varphi$ on $[0,\pi/k]$ will belong to $\mathcal{B}$ and it gives an eigenfunction of \eqref{eq:neuman-eig}. The restriction  of $\varphi$ in $[0,\pi/k]$ has exact one zero. Since $\lambda\in(-\mu_2,0)$, then it must be one of $\mu_i$ for $i\geq 3$. However, any of the eigenfunction will have $i-1$ zeros according to the Sturm-Liouville theory.
\end{proof}

%\begin{claim}
%Negative eigenvalues consist of  $\lambda_i$ of $1\leq i\leq 2k-1$ only.
%\end{claim}
%\begin{proof}[Proof of the claim]
%  Suppose this is not true, then $\lambda_{2k}<0$.  Using Lemma \ref{lem:nodal-two}, one can conclude $\varphi_{2k}$ must have $2k$ nodal sets. \citet[Lemma 7.3]{andrews2003classification} says $\lambda=0$ has eigenspace $span\{h_\theta\}$. It is easy to know $h_\theta(n\pi/k)=0$ for any $n\in \mathbb{Z}$ and it has $2k$ nodal sets. Then $\varphi_{2k}({n\pi/k})\neq 0$ for any $n$, because otherwise Lemma \ref{lem:nodal-diff} will implies the number of nodal sets of $h_\theta$ is more than that of $\varphi_{2k}$. 
%  
%  Moreover, Lemma \ref{lem:nodal-two} also implies $\lambda_{2k}$ must be a simple negative eigenvalue. Since $\varphi_{2k}({n\pi/k})\neq 0$ for any $n$, the proof of Lemma \ref{lem:eig-one} actually implies $\varphi_{2k}$ is symmetric with respect to the reflection of $n\pi/k$ for any $n$.  WLOG, assume $\varphi_{2k}(0)>0$ and $h_{\theta>0}$ on $[0,\pi/k]$. Then $h$ achieve the maximum at $\pi/k$ and minimum at $0$. Consequently $h''(0)>0$
%  Consider 
%\[W(\theta)=\begin{vmatrix} \varphi_{2k}(\theta) & h'(\theta) \\ \varphi'_{2k}(\theta) & h''(\theta)\end{vmatrix}.\]
%It is to know that 
%\[W'(\theta)=\lambda_{2k}\frac{\r[h]}{\alpha h}\varphi_{2k}h_\theta\]
%Integrate from $0$ to $\pi/k$ implies
%\[
%[W(\theta)]|_{0}^{a}=\lambda_{2k}\int_0^{a} \frac{\r[h]}{\alpha h}\varphi_{2k}h_\theta.\]
%Now since $h'(0)=0$,
%\[W(0)=\varphi_{2k}(0)h''(0)>0\]
%\[W(a)=-\varphi_{2k}'(a)h'(a)<0\]
%\end{proof}

We completed the proof the Theorem \ref{thm:main-eig}. From now on, we discuss about  why we may not have $\lambda_k= \lambda_{k+1}$ when $k$ is even. These two eigenvalues are related to the $\eqref{eq:DN-eig}$ and $\eqref{eq:ND-eig}$.

%Using (3) in Proposition \ref{prop:eig}, we have $\mu_1^{DD}=0$ and $\psi_1^{DD}= h_\theta$, because $h_\theta$ does not change sign in $(0,\pi/k)$. 

\begin{lemma} We have $\mu_1^{DN}<0$ and  $\mu_1^{ND}<0$. %If $k$ is even, the first eigenfunctions of \eqref{eq:DN-eig} and \eqref{eq:ND-eig} can be reflected in an appropriate way to get eigenfunctions of $-\mathcal{L}$ on $\mathbb{S}^1$ with respect to the same eigenvalue.
\end{lemma}

\begin{proof} Suppose $\psi^{DN}$ is an eigenfunction corresponding to $\mu_1^{DN}$. Then make an even reflection of $\psi^{DN}$ with respect to $\pi/k$. We get $\psi^{DN}$ is an eigenfunction of 
\[\psi''+\psi=-\frac{1}{\alpha }h^{-1-\frac{1}{\alpha}}(\mu^{DN}+1)\psi,\quad \psi(0)=\psi(2\pi/k)=0\]
Since $\psi^{DN}$ does not change sign in $[0,2\pi]$, it must be the first eigenfunction for the above problem. Note that the reflection of $h_\theta$ also makes an eigenfunction corresponds to $0$ for the above problem. We must have $\mu^{DN}_1<0$.

%One can reflect $\psi^{DN}$ evenly with respect to $2n\pi/k$ and oddly with respect to $(2n+1)\pi/k$ for any $n$. Since $k$ is even, the resulting function is $C^1$ on $\mathbb{S}^1$ and satisfies 
%\[\psi''+\psi=-\frac{\mathfrak{r}[h]}{\alpha h}(\mu^{DN}+1)\psi\]
%this exactly means it is an eigenfunction of $-\mathcal{L}$ corresponds to $\mu^{DN}$. 
The fact of \eqref{eq:ND-eig} can be proved through even reflection with respect to $\theta=0$.
\end{proof}

\begin{remark} 
If $\lambda_k$ is simple, then Lemma \ref{lem:eig-one} says $\varphi_k$ will have $k$ zeros. Moreover, Lemma \ref{lem:eig-one} indicates that the restriction of $\varphi_k$ on $[0,\pi/k]$ will give us a first eigenfunction of \eqref{eq:DN-eig} or \eqref{eq:ND-eig} corresponds to $\mu_1^{DN}$ or $\mu_1^{ND}$. In fact $\lambda_k=\min\{\mu_1^{DN}, \mu_1^{ND}\}$. For the same reason, $\lambda_{k+1}=\max\{\mu_1^{DN}, \mu_1^{ND}\}$. A priori we do not know $\mu_1^{DN}=\mu_1^{ND}$.
\end{remark}

%\begin{proof}[Proof of Claim \ref{clm:simple}]
%Since $\lambda_k$ is simple, then Lemma \ref{lem:eig-one} says $\varphi_k$ will have $k$ zeros. Moreover, the Lemma indicates that the restriction of $\varphi_k$ on $[0,\pi/k]$ will give us a first eigenfunction of \eqref{eq:DN-eig} or \eqref{eq:ND-eig} corresponds to $\mu_1^{DN}$ or $\mu_1^{ND}$. In fact $\lambda_k=\min\{\mu_1^{DN}, \mu_1^{ND}\}$.
%
%Now suppose $\lambda_{k+1}$ is not simple, then  its eigenspace is two dimensional, say it is $span\{\varphi_{k+1},\tilde \varphi_{k+1}\}$. Making some linear combination
%
%  
% The restriction of $\varphi_k$ on $[0,\pi/k]$ will give us an eigenfunction of  
%
%
%\end{proof}

\section{Construction of ancient solutions}\label{sec:construction}

In this section, we construct ancient solutions converging to a shrinker $\Gamma$ after rescaling by using the Morse index $I$ we characterized in Section \ref{sec:spectrum}. Let us denote the Morse index of $\mathcal{L}_{\Gamma}$ by $I(\mathcal{L}_\Gamma)$ . In the section \ref{sec:spectrum}, we showed  $I(\mathcal{L}_{\bar \Gamma_\alpha^k})=2k-1$ and $I(\mathcal{L}_{\bar \Gamma_\alpha^c})=2\lceil 1+1/\alpha\rceil-1$. Again, we shall simply suppress the notation to $I$ and $\mathcal{L}$. One should interpret the following for each case $\Gamma= \bar \Gamma_\alpha^k$ or $\bar \Gamma_\alpha^c$ respectively.

%In this section we will construct some ancient solutions explicitly. Using the negative eigenspace of $-\mathcal{L}$, the first subsection will produce some ancient solution converges exponentially to its backward limit.  Let us denote $I(\mathcal{L})$ the Morse index of $\mathcal{L}_{\Gamma}$. The last section says $I(\bar \Gamma_\alpha^k)=2k-1$ and $I(\bar \Gamma_\alpha^c)=2\lceil 1+1/\alpha\rceil-1$. Again, we shall simply suppress the notation to $I$ and $\mathcal{L}$. One should interpret the following for each case respectively.
%Let us denote denote the Morse index of $\mathcal{L}$. Since $\mathcal{L}$ is determined by ${\Gamma_{\alpha}^k}$, we shall also define the Morse index of ${\Gamma_{\alpha}^k}$ to be that of $\mathcal{L}$. Proposition \ref{prop:eig} says $I=2k-1$.
%Since $-\mathcal{L}$ is self-adjoint on $L_h^2$ and has eigenvalue \eqref{eq:L-eig-h}, 

We begin by considering the inhomogeneous linear PDE
\[\partial_\tau v=\mathcal{L}v+E_\Gamma (v).\]
%Now let us set up the norm and spaces we will use. 
Fix $\beta\in(0,1)$, and for any $f:\mathbb{S}^1\times \mathbb{R}_-\to \mathbb{R}$ we define the seminorm
\[|f(\tau)|_{\C^\beta}=\sup_{(\theta_i,t_i)\in\mathbb{S}^1\times(\tau-1,\tau)}\left\{\frac{|f(\theta_1,t_1)-f(\theta_2,t_2)|}{|\theta_1-\theta_2|^{\beta}+|t_1-t_2|^{\beta/2}}:(\theta_1,t_1)\neq (\theta_2,t_2)\right\}.\]
We use the special symbol $\C$ to denote the parabolic norm in what follows. Notice that we write $\tau$ explicitly in $|f(\tau)|_{\C^\beta}$ to indicate that the parabolic norm is taken on  $\mathbb{S}^1\times(\tau-1,\tau)$.
%on any domain $\Omega\in \mathbb{S}^1\times \mathbb{R}_-$. For the purpose of the following proof, we will choose $\Omega=\mathbb{S}^1\times (\tau-1,\tau)$ for most of time without explicit mention. 
For $l\geq 0$, define the norm
\begin{align}\label{eq:fnorm}
||f(\tau)||_{\C^{l,\beta}}:=\sum_{i+2j\leq l}\sup_{\mathbb{S}^1\times (\tau-1,\tau)}|\partial_\theta^i\partial_t^jf|+\sum_{i+2j=l}|\partial_\theta^i\partial_t^jf|_{\C^\beta}.
\end{align}

For some $\delta>0$, define the norm
\begin{align}\label{def:delta-norm}
||f||_{\C^{l,\beta,\delta}}:=\sup_{\tau\leq 0}\{e^{-\delta\tau}||f||_{\C^{2,\beta}(\mathbb{S}^1\times (\tau-1,\tau))}\}.
\end{align}
Suppose $X^\delta$ is the Banach space equipped with the norm $||f||_{\C^{l,\beta,\delta}}<\infty$.

%\substack{ \\ (\theta_1,t_1),(\theta_2,t_2)\in \Omega}
%When $-\mathcal{L}$ does not have eigenvalue 0, $P_0=\emptyset$.

We fix once and for all an $L^2_h$ orthonormal sequence of eigenfunctions $\varphi_j$ of $-\mathcal{L}$ such that $\mathcal{L}\varphi_j=\lambda_j\varphi_j$ and $(\varphi_j,\varphi_j)_h=1$.
Define $v_j=(v,\varphi_j)_h$ and $P_jv=(v,\varphi_j)_h\varphi_j$. We also define $P_{\leq j}=\sum_{i=1}^jP_i$  and 
\begin{align}\label{def:proj}
P_-=\sum_{j=0}^IP_j,\quad  P_+=\sum_{\{j:\lambda_j>0\}}P_j,\quad P_0=\sum_{\{j:\lambda_j=0\}}P_j.
\end{align}
%where $T>0$ will be chosen and fixed later.
In addition, we define
\[||f||_{L^{2,\delta}}=\sup_{\tau\leq 0}\{e^{-\delta\tau}||f(\cdot,\tau)||_{h}\}.\]
%where $T>0$ will be chosen and fixed later.

For the rest of this section, we will always choose $\delta$ as some positive constant different from $- \lambda_j$ for any $j$. Denote $J=\{j:\lambda_j<-\delta\}\subset\{1,\cdots,I\}$. For example, $J=\emptyset$ if $\delta>\lambda_1=-1-\alpha$. 
\begin{lemma}\label{lem:exp-exist}Fix any $0<\delta\not\in\{-\lambda_j\}_{j=1}^\infty $ and recall the operator $\mathcal{L}$ in \eqref{eq:L-def}.  If $||f||_{L^{2,\delta}}<\infty$,  then the equation
\begin{align*}
\partial_\tau u-\mathcal{L}u=f(\theta,\tau),\quad \text{ on }\quad\mathbb{S}^1\times\mathbb{R}_-
\end{align*}
has  a unique solution $u$ satisfying $||u||_{L^{2,\delta}}<\infty$ and $ P_{j}(u(\cdot, 0))=0$ for $j\in J$. Furthermore, there exists $C=C(\alpha,\beta,\delta)$ such that $||u||_{L^{2,\delta}}\leq C||f||_{L^{2,\delta}}$ and $||u||_{\C^{2,\beta,\delta}}\leq C||f||_{\C^{0,\beta,\delta}}$ hold. 
\end{lemma}

\begin{proof} Recall that $\{\varphi_i\}$ is an orthonormal basis of $L^2_h(\mathbb{S}^1)$ with respect to $( , )_h$. It suffices to solve
\begin{align*}
\partial_\tau u_i+\lambda_iu_i=f_i,\quad \mathbb{S}^1\times \mathbb{R}_-,
\end{align*}
where $u_i=(u,\varphi_i)_h$ and $f_i=(f,\varphi_i)_h$. %We will just prove the case when $f\in Range(P_-)$, the case $f\in Range(P_+)$ can be proved similarly. That is we just need to consider \eqref{eq:ortho-comp} for $\lambda_i<0$. 
% Construct $u=\sum_{j=1}^\infty u_j(\tau)\varphi_j$ by the following 
Denote
\begin{align*}
u_{j}(\tau)&:=\int_{\tau}^{0} e^{\lambda_{j}(s-\tau)} f_{j}(s) d s, \quad j\in J, \\
u_{j}(\tau)&:=\int_{-\infty}^{\tau} e^{\lambda_{j}(s-\tau)} f_{j}(s) d s, \quad j\in J^c=\mathbb{Z}_+\backslash J.
\end{align*}
Notice the integral on the RHS is well-defined because $|f_j(s)|\leq ||f(\cdot,s)||_h\leq ||f||_{L^{2,\delta}}e^{\delta s}$.
  Define $u(\cdot,\tau)=\sum_{j=1}^\infty u_j(\tau)\varphi_j(\cdot)$. It is easy to see $P_{j}(u(\cdot,0))=0$ for any $j\in J$. 
  
  Choose $\delta'$ and $\delta''$ satisfying  $\max_{j\in J}\{\lambda_j\}<-\delta'<-\delta<-\delta''<\min_{j\in J^c}\{\lambda_j\}$. Note that for $j\in J$
 \begin{align*}
 u_j^2(\tau)\leq \int_\tau^0 e^{2(\lambda_j+\delta')(s-\tau)}ds\int_\tau^0 e^{-2\delta'(s-\tau)}|f_j|^2ds\leq C\int_\tau^0 e^{-2\delta'(s-\tau)}|f_j|^2ds
 \end{align*}
 and for $j\in J^c$
  \begin{align*}
 u_j^2(\tau)\leq \int_{-\infty}^\tau e^{2(\lambda_j+\delta'')(s-\tau)}ds\int_{-\infty}^\tau e^{-2\delta''(s-\tau)}|f_j|^2ds\leq C\int_{-\infty}^\tau e^{-2\delta''(s-\tau)}|f_j|^2ds.
 \end{align*}
 Combining the above two inequalities and using $|f_j(s)|\leq ||f||_{L^{2,\delta}}e^{\delta s}$, one obtains
 \begin{align*}
  ||u(\cdot,\tau)||_{h}^2=\sum_{j}u_j^2(\tau)\leq& C\int_\tau^0 e^{-2\delta'(s-\tau)}|f_j|^2ds+C\int_{-\infty}^\tau e^{-2\delta''(s-\tau)}|f_j|^2ds\\
  \leq &C||f||_{L^{2,\delta}}^2e^{2\delta\tau}.
 \end{align*}
%  Let $\lambda_{max}=\max_{j\in J}\{\lambda_j\}$ and $\lambda^c_{min}=\min_{j\in J^c}\{\lambda_j\}$. Then
%  \begin{align*}
%  ||u(\cdot,\tau)||_{h}^2=& \sum_{j\in J}\left(\int_\tau^0e^{\lambda_j(s-\tau)}|f_j|(s)ds\right)^2+\sum_{j\in J^c}\left(\int^\tau_{-\infty}e^{\lambda_j(s-\tau)}|f_j|(s)ds\right)^2\\
%  \leq&\sum_{j\in J}\left(\int_\tau^0e^{\lambda_{max}(s-\tau)}|f_j|(s)ds\right)^2+\sum_{j\in J^c}\left(\int^\tau_{-\infty}e^{\lambda_{min}^c(s-\tau)}|f_j|(s)ds\right)^2\\
%  = &\left|\left|\int_\tau^0e^{\lambda_{max}(s-\tau)}(\textstyle\sum\limits_{j\in J}P_{j}f) ds\right|\right|_{h}^2+\left|\left|\int_{-\infty}^\tau e^{\lambda^c_{min}(s-\tau)}(\textstyle\sum\limits_{j\in J^c}P_{j}f)ds\right|\right|_{h}^2
%  \end{align*}
%where in the last step we made use of the Parseval identity. Taking square roots and using our assumption on $f$ show
%\begin{align*}
%||u(\cdot,\tau)||_{h}\leq& \int_\tau^0e^{\lambda_{max}(s-\tau)}||f(\cdot,s)||_{h}ds+\int_{-\infty}^\tau e^{\lambda_{min}^c(s-\tau)}||f(\cdot,s)||_{h}ds\\
%\leq &C(\delta,\lambda_{max},\lambda_{min}^c) ||f||_{L^{2,\delta}}e^{\delta \tau}.
%\end{align*} 
Therefore $||u||_{L^{2,\delta}}\leq C||f||_{L^{2,\delta}}$.
%%\[||u||_{L^{2,\delta}}\leq C||f||_{L^{2,\delta}}.\]

Let's establish $\C^{2,\beta,\delta}$ bounds. By the interior parabolic Schauder estimates (for instance, see \cite[(C.6)]{choi2019ancient}), we have that for any $\tau\leq 0$, 
\[||u||_{\C^{2,\beta}(\mathbb{S}^1\times (\tau-1,\tau))}\leq C\left(||u||_{L^2(\mathbb{S}^1\times (\tau-2,\tau))}+||f||_{\C^{0,\beta}(\mathbb{S}^1\times(\tau-2,\tau))}\right).\]
Multiplying by $e^{\delta\tau}$ and taking the supremum over $\tau\leq 0$ yield
\begin{align*}
||u||_{\C^{2,\beta,\delta}}\leq C(||u||_{L^{2,\delta}}+||f||_{\C^{0,\beta,\delta}})\leq C(||f||_{L^{2,\delta}}+||f||_{\C^{0,\beta,\delta}})\leq C||f||_{\C^{0,\beta,\delta}}.
\end{align*}

%(2) If $f\in Range(P_+)$, then $f=\sum_{\{i:\lambda_i>0\}}f_i\varphi_i$. Since $\lambda_i>0$ and $f\in L^2_q$, the general solution for \eqref{eq:ortho-comp} is 
%\[u_i(\tau)=c_ie^{-\lambda_i\tau}+e^{-\lambda_i\tau}\int^\tau_{-\infty}e^{\lambda_is}f_i(s)ds\]
%where $c_i$ are free constant. In order to have $u_i\in L_q^2$, we must choose $c_i=0$.
%\[u=\sum_{\{i:\lambda_i>0\}}\left(e^{-\lambda_i\tau}\int^\tau_{-\infty}e^{\lambda_is}f_i(s)ds\right)\varphi_i\]
%The rest of proof is almost the same as (1).
\end{proof}

We shall use contraction mapping theorem and the above lemma repeatedly to construct ancient solutions. Let us introduce some necessary notations. For any $\boldsymbol{a}=(a_1,\cdots,a_I)\in \mathbb{R}^I$, denote $|\boldsymbol{a}|=(\sum_{i=1}^Ia_i^2)^{\frac12}$.
We introduce auxiliary operators  which maps any integer set $J\subset\{1,2,\cdots,I\}$ to functions,  
\begin{align*}
    \iota^J:\mathbb{R}^{I}&\to L_h^2\times \mathbb{R}_-,\\
    \iota^J(\boldsymbol{a})&:=\sum_{j\in J} a_je^{-\lambda_j\tau}\varphi_j.
\end{align*}

Denote $L=\lfloor\lambda_1/\lambda_I\rfloor$. For each $l=1,\cdots,L$,  define 
\[J^{(l)}=\{m:(l+1)\lambda_I< \lambda_m\leq l\lambda_I\}.\]
 Then $\cup_{l=1}^LJ^{(l)}$ is a partition of $\{1,\cdots,I\}$ according to the negative eigenvalues of $\mathcal{L}_\Gamma$. Choose $\delta_l$ satisfying 
 \[\max\{\lambda_j:j\in J^{(l+1)}\}\leq (l+1)\lambda_I<-\delta_l<\min\{\lambda_j:j\in J^{(l)}\}.\]
%\[\max\{\lambda_j:j\in J^{(l+1)}\}<\delta_l<\min\{\lambda_j:j\in J^{(l)}\}\]
%$\delta_l=(l+1)\lambda_I+\varepsilon$ with $\varepsilon$ small enough such that $\delta_l<\min\{\lambda_m:\lambda_m>(l+1)\lambda_I\}$. 
Write $X^{(l)}=X^{\delta_l}$, $\iota^{(l)}=\iota^{J^{(l)}}$,  and $P^{(l)}=\sum_{j:\lambda_j<-\delta_l}P_{j}$ for simplicity. In what follows, we will use the symbol $\lesssim$ for inequalities that hold up to multiplicative constants that may depend on $\alpha,h$.

Here is the main result of this section 
\begin{theorem}\label{thm:exist-type1} Let $L=\lfloor\lambda_1/\lambda_I\rfloor$\footnote{$\lfloor x\rfloor$ means the greatest integer less than or equal to $x$}. There exists some $\varepsilon_0>0$ satisfying the following significance. Given $\boldsymbol{a}=(a_1,\cdots,a_{I})\in \mathbb{R}^I$ with $|\boldsymbol{a}|<\varepsilon_0$, there exist  a set of functions $\{v^{(l)}\}_{l=1}^L$  uniquely determined and depending continuously on $\boldsymbol{a}$ such that for each $l=1,\cdots, L$, we have $v^{(l)}-\iota^{(l)}(\boldsymbol{a})\in X^{(l)}$, $P^{(l)}(v^{(l)}-\iota^{(l)}(\boldsymbol{a}))(\cdot,0)=0$, and $\sum_{j=1}^lv^{(j)}$ is an ancient solution of \eqref{intro:eq:v-flow} for $(-\infty,0]$.  More importantly
\begin{align}\label{eq:layer}
\lim_{\tau\to -\infty }e^{\lambda_m\tau}(v^{(l)}(\cdot,\tau),\varphi_m)_h=a_m, \quad m\in J^{(l)}.
\end{align}
\end{theorem}
Let us first prove a proposition which will be needed in the proof of Theorem \ref{thm:exist-type1}.
\begin{proposition}\label{prop:error_expansion}
There exists some constants $C=C(\alpha,h)$ and $\epsilon=\epsilon(\alpha,h)>0$ such that if $|v_{\theta\theta}+v|\leq \epsilon$ then
\begin{align}
| E_{\Gamma}(v)|&\leq C|v_{\theta\theta}+v|^2,\label{eq:Ec0}\\
|E_\Gamma(v)(\tau)|_{\C^\beta}&\leq C|(v_{\theta\theta}+v)(\tau)|_{\C^\beta}|(v_{\theta\theta}+v)(\tau)|_{\C^0}\label{eq:Ecbeta}.
\end{align}
Moreover, if $u,v$ satisfy  $|u_{\theta\theta}+u|+|v_{\theta\theta}+v|\leq \epsilon$ then
\begin{align*}
&|E_\Gamma(u)(\tau)-E_\Gamma(v)(\tau)|_{\C^\beta}\\
\leq& C|((u-v)_{\theta\theta}+u-v)(\tau)|_{\C^\beta}\left[|(u_{\theta\theta}+u)(\tau)|_{\C^0}+|(v_{\theta\theta}+v)(\tau)|_{\C^0}\right]
\end{align*}
\end{proposition}
\begin{proof}
By the definition $E_\Gamma$ in \eqref{intro:eq:v-flow}, we obtain
\[E_\Gamma(v)=-h{\left(1+h^{\frac{1}{\alpha}}(v_{\theta\theta}+v)\right)^{-\alpha}}+h-\alpha h^{1+\frac{1}{\alpha}}(v_{\theta\theta}+v).\]
Using the Taylor expansion of $(1+x)^{-\alpha}$, it is easy to know $|(1+x)^{-\alpha}-1+\alpha x|\leq C(\alpha)x^2$ whenever $|x|<1/2$. Therefore we have 
\[|E_\Gamma(v)|\leq C(\alpha,h)|v_{\theta\theta}+v|^2\]
whenever $|v_{\theta\theta}+v|<\frac12 h^{-1/\alpha}$. 

Our second conclusion follows from the following observations
\[|(1+x)^{-\alpha}+\alpha x-(1+y)^{-\alpha}-\alpha y|\leq C(\alpha)(x-y)^2,\quad |x|+|y|<\frac12\]
and consequently for any $t_1,t_2\in [\tau-1,\tau]$
\begin{align*}
&|E_\Gamma(v)(\theta_1,t_1)-E_\Gamma(v)(\theta_2,t_2)|\\
\leq &C(\alpha,h)|(v_{\theta\theta}+v)(\theta_1,t_1)-(v_{\theta\theta}+v)(\theta_2,t_2)|^2\\
\leq & C(\alpha,h) (|\theta_1-\theta_2|^\beta+|t_1-t_2|^{\beta/2})|(v_{\theta\theta}+v)(\tau)|_{\C^\beta}|(v_{\theta\theta}+v)(\tau)|_{\C^0}.
\end{align*}
The estimates of $|E_\Gamma(u)(\tau)-E_\Gamma(v)(\tau)|_{\C^{\beta}}$ can be proved similarly.
\end{proof}

Now we can prove the main theorem of this section.
\begin{proof}[Proof of Theorem \ref{thm:exist-type1}]
We shall find all $v^{(l)}$ by the induction. First, we notice that $\iota^{(1)}$ is an  ancient solution to the linear equation $\partial_\tau v=\mathcal{L}v $. Therefore, to find $v^{(1)}$,  we assume $v^{(1)}=w+\iota^{(1)}(\boldsymbol{a})$ for some $w$ to be determined. Then \eqref{intro:eq:v-flow} is equivalent to 
\begin{align}\label{eq:w-cont}
\partial_\tau w=\mathcal{L}w+E (w+\iota^{{(1)}}(\boldsymbol{a}))
\end{align}
Here and in the following, we shall write $E(v)=E_\Gamma(v)$ and $\r[v]=v_{\theta\theta}+v$ for short.
\begin{claim} There exists small $\varepsilon_0$ such that if $||w||_{ X^{(1)}}+|\boldsymbol{a}|<\varepsilon_0$ then 
\begin{align}
||E (w+\iota^{(1)}(\boldsymbol{a}))||_{\C^{0,\beta,\delta_1}}&\lesssim ||w||_{X^{(1)}}^2+|\boldsymbol{a}|^2,\label{eq:v1-1}\\
||E (w_1+\iota^{(1)}(\boldsymbol{a}))-E (w_2+\iota^{(1)}(\boldsymbol{a}))||_{\C^{0,\beta,\delta_1}}&\lesssim \varepsilon_0||w_1-w_2||_{X^{(1)}}.\label{eq:v1-2}
\end{align}
\end{claim}
In fact, one can easily derive from \eqref{eq:Ec0} and \eqref{eq:Ecbeta} that there exists $\varepsilon_0>0$ and some constant $C(\alpha,\beta,h)$ such that
   \begin{align*}
    ||E (v)(\tau)||_{\C^{0,\beta}}\leq C||\r[v](\tau)||_{\C^{0,\beta}}^2
    \end{align*}
    provided $||\r[v](\tau)||_{\C^0}<\varepsilon_0$. Furthermore
    \begin{align}
    ||E (v_1)(\tau)-&E (v_2)(\tau)||_{\C^{0,\beta}}\label{eq:R-2}\\
    \leq &C(||\r[v_1]\left(\tau)||_{\C^{0,\beta}}+||\r[v_2](\tau)||_{\C^{0,\beta}}\right)||\r[v_1-v_2](\tau)||_{\C^{0,\beta}}\notag
   \end{align}
   provided $||\r[v_1](\tau)||_{\C^{0}}+||\r[v_2](\tau)||_{\C^{0}}<\varepsilon_0$. 
%note that \eqref{eq:R-exp} implies $E $ is an absolutely convergent series if $|\r[v]|$ small and $E (v)\sim \r[v]^2$. 
Therefore 
\begin{align*}
||E (w+\iota^{(1)}(\boldsymbol{a}))(\tau)||_{\C^{0,\beta}}\lesssim& ||\r[w](\tau)+\r[\iota^{(1)}(\boldsymbol{a})](\tau)||^2_{\C^{0,\beta}}\\
 \lesssim& ||w(\tau)||^2_{\C^{2,\beta}}+|\boldsymbol{a}|^2e^{-2\lambda_I\tau}.
\end{align*}
Recall our definition of the norm \eqref{def:delta-norm}, multiplying the above inequality by $e^{\delta_1\tau}$ and noticing $-\delta_1>2\lambda_I$, one can get \eqref{eq:v1-1} holds. Moreover
\begin{align*}
&||E (w_1+\iota^{(1)}(\boldsymbol{a}))-E (w_2+\iota^{(1)}(\boldsymbol{a}))||_{\C^{0,\beta}(\mathbb{S}^1\times(\tau-1,\tau))}\\
\lesssim& (||\r[w_1](\tau)||_{\C^{0,\beta}}+||\r[w_2](\tau)||_{\C^{0,\beta}}+|\boldsymbol{a}|e^{-\lambda_I\tau})||(\r[w_1]-\r[w_2])(\tau)||_{\C^{0,\beta}}
\end{align*}
which implies \eqref{eq:v1-2} holds. Thus, the claim is proved.

Define a map $S:\{f\in X^{(1)}:||f||_{X^{(1)}}<\varepsilon_0\}\to X^{(1)}$ by $S(w)=u$ where $u$ is the solution of 
\[\partial_\tau u-\mathcal{L}u=E (w+\iota^{(1)}(\boldsymbol{a}))\quad \text{on }\quad \mathbb{S}^1\times\mathbb{R}_-\] 
with $P^{(1)}(u(\cdot,0))=0$. By Lemma \ref{lem:exp-exist}, such $u\in X^{(1)}$  is unique, so $S(w)$ is well-defined. Moreover, Lemma \ref{lem:exp-exist} says , 
\[\quad ||u||_{X^{(1)}}\lesssim ||w||_{X^{(1)}}^2+|\boldsymbol{a}|^2\leq \varepsilon_0^2+|\boldsymbol{a}|^2\]
and
\[||S(w_1)-S(w_2)||_{X^{(1)}}\lesssim (\varepsilon_0+|\boldsymbol{a}|)||w_1-w_2||_{X^{(1)}}. \]
 Choosing $\varepsilon_0$ small enough, $S$ will be a contraction mapping on $\{f\in X^{(1)}:||f||_{X^{(1)}}<\varepsilon_0\}$. Therefore it has a fixed point $w$ which solves \eqref{eq:w-cont}. Since $w\lesssim e^{\delta_1 \tau}$, we have
\[\lim_{\tau\to -\infty}e^{\lambda_m\tau}(v^{(1)},\varphi_m)_h=\lim_{\tau\to -\infty}e^{\lambda_m\tau}(\iota^{(1)}(\boldsymbol{a}),\varphi_m)_h=a_m,\quad m\in J^{(1)}.\]

 Therefore, we have found $v^{(1)}$ and \eqref{eq:layer} is true for $l=1$.

Suppose that we have found $v^{1}$ up to $v^{(l)}$, and \eqref{eq:layer} is established up to $l$ by the induction. If $J^{(l+1)}= \emptyset$, let $v^{(l+1)}=0$. Obviously the theorem still holds for such $v^{(l+1)}$.
%such that for each $v^{(l)}\in X^{(l)}$, and it is the unique ancient solution of 
%\[\partial_\tau v^{(l)}=\mathcal{L}v^{(l)}+E (\sum_{j=1}^{l}v^{(j)})-E (\sum_{j=1}^{l-1}v^{(j)}).\]
%(here we define $v^{(0)}=0$) with $P_{\leq J^{(l)}}(v^{l}(\cdot,0))=0$, 
%\[\lim_{\tau\to -\infty}e^{\lambda_m\tau}(v^{(l)},\varphi_m)=a_m,\quad m\in J^{(l)}.\]
If $J^{(l+1)}\neq \emptyset$, then we can find $v^{(l+1)}$ by the following process. Let $v^{(l+1)}=w+\iota^{(l+1)}(\boldsymbol{a})$. Since we require $\sum_{j=1}^{l+1}v^{(j)}$ is an ancient solution of \eqref{intro:eq:v-flow}, it suffices to find $w\in X^{(l+1)}$ such that 
%then it implies $v^{(l+1)}$ solves
%\[\partial_\tau v^{(l+1)}=\mathcal{L}v^{(l+1)}+E \left(\sum_{j=1}^{l+1}v^{(j)}\right)-E \left(\sum_{j=1}^{l}v^{(j)}\right).\]
%Similarly let $v^{(l+1)}=w+\iota^{(l+1)}(\boldsymbol{a})$ for some  $w\in X^{(l+1)}$. We want to find $w$ such that 
\begin{align}\label{eq:l+1}
\partial_\tau w=\mathcal{L}w+E ^{(l+1)}(w)
\end{align}
where 
\begin{align}\label{eq:l+1-2}
E ^{(l+1)}(w)=E \left(w+\iota^{(l+1)}(\boldsymbol{a})+\sum_{j=1}^{l}v^{(j)}\right)-E \left(\sum_{j=1}^{l}v^{(j)}\right).
\end{align}
%Using, one can obtain the following claim
\begin{claim}There exists $\varepsilon_0$ small such that if $||w||_{ X^{(l+1)}}+|\boldsymbol{a}|<\varepsilon_0$ then
\begin{align}
||E ^{(l+1)}(w)||_{\C^{0,\beta,\delta_{l+1}}}&\lesssim ||w||^2_{X^{(l+1)}}+|\boldsymbol{a}|^2,\label{eq:vl-11}\\
||E ^{(l+1)}(w_1)-E ^{(l+1)}(w_2)||_{\C^{0,\beta,\delta_{l+1}}}&\lesssim \varepsilon_0||w_1-w_2||_{X^{(l+1)}}.\label{eq:vl-21}
\end{align} 
\end{claim}
In fact, using \eqref{eq:R-2}, for $\r[w_1]$ and $\r[w_2]$ small
\begin{align*}
||[E (w_1)&-E (w_2)](\tau)||_{\C^{0,\beta}}\\
\lesssim& (||\r[w_1](\tau)||_{\C^{0,\beta}}+||\r[w_2](\tau)||_{\C^{0,\beta}})||[\r[w_1]-\r[w_2]](\tau)||_{\C^{0,\beta}}.
\end{align*}
This implies
\begin{align*}
||E ^{(l+1)}(w)(\tau)||_{\C^{0,\beta}}\lesssim& (||w(\tau)||_{\C^{2,\beta}}+|\boldsymbol{a}|e^{-\lambda_I\tau})(||w(\tau)||_{\C^{2,\beta}}+|\boldsymbol{a}|e^{-(l+1)\lambda_I\tau})\\
\lesssim&e^{-(l+2)\lambda_I\tau}(||w||_{\C^{2,\beta,-\lambda_I}}+|\boldsymbol{a}|)(||w||_{\C^{2,\beta,-(l+1)\lambda_I}}+|\boldsymbol{a}|).
\end{align*}
Recalling \eqref{def:delta-norm} and $(l+2)\lambda_I<-\delta_{l+1}<(l+1)\lambda_I$, one can see that \eqref{eq:vl-11} holds. The proof of \eqref{eq:vl-21} can be derived similarly.

Define a map $S:\{f\in X^{(l+1)}:||f||_{X^{(l+1)}}<\varepsilon_0\}\to X^{(l+1)}$ by $S(w)=u$ where $u$ is the unique solution of 
\[\partial_\tau u-\mathcal{L}u=E ^{(l+1)}(w)\quad \text{on }\quad\mathbb{S}^1\times\mathbb{R}_-\] 
with $P^{(l+1)}(u(\cdot,0))=0$. Taking $\varepsilon_0$ small enough, $S$ is a contraction mapping. Note that \eqref{eq:layer} is satisfied for $m\in J^{(l+1)}$. The existence of $v^{(l+1)}$ is established. 

Finally, the uniqueness and continuity of $v^{(l)}$ also follow from the contraction mapping theorem.
\end{proof}

According to Theorem \ref{thm:exist-type1}, we can have $I$-parameter family of ancient solutions, but only $I-3$ of them are important, because the first three of them depends on the time and space center we choose in the renormalization \eqref{eq:normalized_flow}. That is, one can always just shift the non-rescaled ancient solution by time and space to edit the first three parameters.

\begin{proposition}\label{prop:rescaling_center}
Let $\Gamma_t$ be an $\alpha$-CSF asymptotic to $\Gamma$ after rescaling where $\Gamma=\bar\Gamma_{\alpha}^k$ or $\bar\Gamma_{\alpha}^c$, and let $\bar u$ denote the support function of the rescaled flow $\bar \Gamma_\tau$. Then, given $B=(b_1,b_2,b_3) \in\mathbb{R}^3$ the ancient flow 
\begin{equation}
\Gamma^{B}_t= \Gamma_{t+ b_1}+(b_2,b_3)\subset \mathbb{R}^2
\end{equation}
satisfies 
\begin{align*}
&\bar{u}^{B}(\theta,\tau)-\bar u(\theta,\tau)\\
&=(1+\alpha)^{-\frac{1}{\alpha+1}}[b_2\cos\theta+b_3\sin\theta]e^\tau+\frac{b_1}{1+\alpha}e^{(1+\alpha)\tau}h+o(e^{(1+\alpha)\tau})
\end{align*}
where $\bar u^{B}$ denotes the support function of the rescaled flow $\bar \Gamma_\tau^B$. Consequently if $\bar\Gamma_\tau$ is constructed from $\boldsymbol{a}$, then $\bar \Gamma_\tau^B$ is constructed from using 
\[\boldsymbol{a}+\left(\frac{b_1}{1+\alpha},(1+\alpha)^{-\frac{1}{\alpha+1}}b_2,(1+\alpha)^{-\frac{1}{\alpha+1}}b_3,0,\cdots,0\right).\]
\end{proposition}

\begin{proof}Our assumption implies
\[u^{B}(\theta,t)=u(\theta,t+b_1)+b_2\cos\theta+b_3\sin\theta\]
where $u^{B}$ and $u$ are the support functions of $\Gamma_t^{B}$ and $\Gamma_t$. It follows from \eqref{eq:normalized_flow} that
\begin{align}
\bar{u}^{B}(\theta,\tau)=&(1+\alpha)^{-\frac{1}{\alpha+1}}e^\tau u^{B}(\theta,-e^{-(1+\alpha)\tau})\notag\\
=&(1+\alpha)^{-\frac{1}{\alpha+1}}e^{\tau} \left[u(\theta,-e^{-(1+\alpha)\tau}+b_1)+b_2\cos\theta+b_3\sin\theta\right]\notag\\
=&e^{\tau-\tau_1}\bar u(\theta,\tau_1)+(1+\alpha)^{-\frac{1}{\alpha+1}}[b_2\cos\theta+b_3\sin\theta]e^\tau\label{eq:tau1}
\end{align}
where 
\[\tau_1=\frac{-1}{1+\alpha}\log(e^{-(1+\alpha)\tau}-b_1).\]
As $\tau\to -\infty$, we have $\tau_1=\tau+(1+\alpha)^{-1}b_1e^{(1+\alpha)\tau}+o(e^{(1+\alpha)\tau})$. Consequently $e^{\tau-\tau_1}=1+(1+\alpha)^{-1}b_1e^{(1+\alpha)\tau}+o(e^{(1+\alpha)\tau})$.

Since $\bar\Gamma_\tau$ converges to some self-shrinker $\Gamma$ with support function $h$,  $\bar u(\theta,\tau)\to h$ as $\tau\to -\infty$. Plugging the above information of $\tau_1$ to \eqref{eq:tau1}, one gets the conclusion.
%TO DO : Find $c_1,c_2$ and prove the proposition.
\end{proof}

\begin{proof}[Proof of Theorem \ref{intro:thm:exist1}]
It follows from Theorem \ref{thm:exist-type1} that there exists a map 
\begin{align*}
\mathcal{S}:B_{\varepsilon_0}(\subset \mathbb{R}^I)&\to C^{2,\beta}(\mathbb{S}^1\times(-\infty,0])\\
\boldsymbol{a}=(a_1,\cdots,a_I)&\mapsto\sum_{j=1}^{L}v^{(j)}
\end{align*}
such that $\mathcal{S}(\boldsymbol{a})=\sum_{j=1}^{L}v^{(j)}$ is an ancient solution of \eqref{intro:eq:v-flow}. 

Suppose that $\boldsymbol{a},\boldsymbol{b}\in \mathbb{R}^{I}$ satisfy $a_{k}-b_{k} \neq 0$ and  $a_i-b_i = 0$ for all $i>k$. Assume $\lambda_k\in J^{(l+1)}$ for some $l$. Then for any $j\in\{1,\cdots,l\}$, careful tracking the proof of Theorem \ref{thm:exist-type1} shows that $v^{(j)}_{\boldsymbol{a}}=v^{(j)}_{\boldsymbol{b}}$, because of the uniqueness of them obtained through the contraction mapping theorem. In the step to find $v^{(l+1)}_{\boldsymbol{a}}=\iota^{(l+1)}(\boldsymbol{a})+w_{\boldsymbol{a}}$ and $v^{(l+1)}_{\boldsymbol{a}}=\iota^{(l+1)}(\boldsymbol{b})+w_{\boldsymbol{b}}$,  we recall that
% they satisfy \eqref{eq:l+1} and \eqref{eq:l+1-2} respectively. Since 
$w_{\boldsymbol{a}}, w_{\boldsymbol{b}}\in X^{(l+1)} $. Therefore 
%\[v^{(l+1)}_{\boldsymbol{a}}-v^{(l+1)}_{\boldsymbol{b}}=(a_k)\]
% following the proof of \eqref{eq:vl-21} one obtains
%\[||E^{(l+1)}(w_{\boldsymbol{a}})-E^{(l+1)}(w_{\boldsymbol{b}})||_{\C^{0,\beta,\delta_{l+1}}}\lesssim e^{-\lambda_{k+1}\tau}.\]
%\[\partial_\tau(w_{\boldsymbol{a}}-w_{\boldsymbol{b}})=\mathcal{L}(w_{\boldsymbol{a}}-w_{\boldsymbol{b}})+E^{(l+1)}(w_{\boldsymbol{a}})-E^{(l+1)}(w_{\boldsymbol{b}})\]

% then $\mathcal{S}$ satisfies
\begin{align*}
 \mathcal{S}(\boldsymbol{a})(\theta,\tau)-\mathcal{S}(\boldsymbol{b})(\theta,\tau)&=v^{(l+1)}_{\boldsymbol{a}}-v^{(l+1)}_{\boldsymbol{b}}+O(e^{-\delta_{l+1}\tau})\\
&= (a_{k}-b_{k}) e^{-\lambda_{k}\tau}\varphi_k(\theta) +O(e^{-\lambda_{k-1}\tau})+O(e^{-\delta_{l+1}\tau})\\
&= (a_{k}-b_{k}) e^{-\lambda_{k}\tau}\varphi_k(\theta) +o(e^{-\lambda_{k}\tau})\end{align*}
when $\lambda_{k-1}<\lambda_k$, and
\begin{align*}
 \mathcal{S}(\boldsymbol{a})(\theta,\tau)-\mathcal{S}(\boldsymbol{b})(\theta,\tau)=e^{-\lambda_{k}\tau}\sum_{i=k-1}^{k}(a_{i}-b_{i}) \varphi_i(\theta) +o(e^{-\lambda_{k}\tau})
\end{align*}
when $\lambda_{k-1}=\lambda_k$, where $\varphi_i$ are eigenfunctions of $\mathcal{L}_\Gamma$ with the eigenvalue $\lambda_i$ and $\langle \varphi_i,\varphi_j\rangle_{L^2_h}=\delta_{ij}$.

Proposition \ref{prop:rescaling_center} says it suffices to have the map for $\boldsymbol{a}=(0,0,0,a_3,\cdots,a_I)$, because the other ancient solutions can be generated by these ones from a different choice of time and space center in the renormalization. Removing the first three zeros of $\boldsymbol{a}$ and reindexing each components, we abuse the notation by still denoting $\boldsymbol{a}=(a_1,\cdots,a_{I-3})$. Therefore we have a map $\mathcal{S}:B_{\epsilon_0}(0)\subset \mathbb{R}^{I-3}\to C^{2,\beta}(\mathbb{S}^1\times(-\infty,0])$ with the desired property.
\end{proof}

Notice that in Theorem \ref{thm:exist-type1} there is a restriction $|\boldsymbol{a}|<\varepsilon_0$. It is possible to get around this by translating in the $\tau$ (which is equivalent to parabolic scaling in the corresponding non-rescaled ancient solution). However, one has to pay the price that these ancient solutions may not live up to $\tau=0$.
\begin{theorem}
 There exists a continuous map $\mathcal{S}$ (define in \eqref{eq:S(a)}) which maps any $\boldsymbol{a}\in \mathbb{R}^I$ to $\C^{2,\beta,-\lambda_I}(\mathbb{S}^1\times(-\infty, T({\boldsymbol{a}}))$ such that $\mathcal{S}(a)$ is an ancient solution of \eqref{intro:eq:v-flow} on $(-\infty, T({\boldsymbol{a}})]$, where $T(\boldsymbol{a})$ is defined in \eqref{eq:T(a)}. Moreover there is a unique decomposition $\mathcal{S}(a)=\sum_{l=1}^L v^{(l)}$ with 
\[\lim_{\tau\to -\infty }e^{\lambda_m\tau}(v^{(l)}(\cdot,\tau),\varphi_m)_h=a_m, \quad \forall \,m \text{ satisfying }\lambda_m\in((l+1)\lambda_I, l\lambda_I]\]
for any $l=1,\cdots,L$.
\end{theorem}

\begin{proof}
 For any $\boldsymbol{a}\in \mathbb{R}^I$, let 
\begin{align}\label{eq:T(a)}
T(\boldsymbol{a})=\frac{1}{2(1+\alpha)}\max\{\log\frac{\varepsilon_0}{|\boldsymbol{a}|^2},0\}
\end{align}
then 
$\sum_{i=1}^I e^{-2\lambda_m T}a_m^2<\varepsilon_0$. Then by the previous theorem, one can find $\{v^{(l)}\}_{l=1}^L$ such that $\sum_{l=1}^Lv^{(l)}$ is an ancient solution of \eqref{intro:eq:v-flow} on $\mathbb{S}^1\times(-\infty,0]$ and
\[\lim_{\tau\to -\infty }e^{\lambda_m\tau}(v^{(l)}(\cdot,\tau),\varphi_m)_h=e^{-\lambda_mT}a_m, \quad m\in J^{(l)}.\]
So we define a map $\mathcal{S}$ by translating $v^{(l)}$
\begin{align}\label{eq:S(a)}
\mathcal{S}(a)(\cdot,\tau)=\sum_{l=1}^L v^{(l)}(\cdot,\tau-T).
\end{align}
One can easily verify that $\mathcal{S}(a)$ is an ancient solution of \eqref{intro:eq:v-flow} on $\mathbb{S}^1\times(-\infty,-T]$ and 
\[\lim_{\tau\to -\infty }e^{\lambda_m\tau}(v^{(l)}(\cdot,\tau-T),\varphi_m)_h=a_m, \quad m\in J^{(l)}.\]
\end{proof}

\bibliographystyle{plainnat}
\bibliography{CSF-ref}

\begin{thebibliography}{26}
\providecommand{\natexlab}[1]{#1}
\providecommand{\url}[1]{\texttt{#1}}
\expandafter\ifx\csname urlstyle\endcsname\relax
  \providecommand{\doi}[1]{doi: #1}\else
  \providecommand{\doi}{doi: \begingroup \urlstyle{rm}\Url}\fi

\bibitem[Andrews(2002)]{andrews2002non}
Ben Andrews.
\newblock Non-convergence and instability in the asymptotic behaviour of curves
  evolving by curvature.
\newblock \emph{Communications in Analysis and Geometry}, 10\penalty0
  (2):\penalty0 409--449, 2002.

\bibitem[Andrews(2003)]{andrews2003classification}
Ben Andrews.
\newblock Classification of limiting shapes for isotropic curve flows.
\newblock \emph{Journal of the American mathematical society}, 16\penalty0
  (2):\penalty0 443--459, 2003.

\bibitem[Andrews et~al.(2016)Andrews, Guan, and Ni]{andrews2016flow}
Ben Andrews, Pengfei Guan, and Lei Ni.
\newblock Flow by powers of the gauss curvature.
\newblock \emph{Advances in Mathematics}, 299:\penalty0 174--201, 2016.

\bibitem[Angenent(1992)]{angenent1992doughnuts}
Sigurd Angenent.
\newblock Shrinking doughnuts.
\newblock In \emph{Nonlinear diffusion equations and their equilibrium states,
  3}, pages 21--38. Springer, 1992.

\bibitem[Angenent et~al.(2019{\natexlab{a}})Angenent, Brendle, Daskalopoulos,
  and Sesum]{angenent2019unique-B}
Sigurd Angenent, Simon Brendle, Panagiota Daskalopoulos, and Natasa Sesum.
\newblock Unique asymptotics of compact ancient solutions to three-dimensional
  ricci flow.
\newblock \emph{arXiv preprint arXiv:1911.00091}, 2019{\natexlab{a}}.

\bibitem[Angenent et~al.(2019{\natexlab{b}})Angenent, Daskalopoulos, and
  Sesum]{angenent2019unique}
Sigurd Angenent, Panagiota Daskalopoulos, and Natasa Sesum.
\newblock {Unique asymptotics of ancient convex mean curvature flow solutions}.
\newblock \emph{Journal of Differential Geometry}, 111\penalty0 (3):\penalty0
  381--455, 2019{\natexlab{b}}.

\bibitem[Angenent et~al.(2018)Angenent, Daskalopoulos, and
  Sesum]{angenent2018uniqueness}
Sigurd~B Angenent, Panagiota Daskalopoulos, and Natasa Sesum.
\newblock Uniqueness of two-convex closed ancient solutions to the mean
  curvature flow.
\newblock \emph{arXiv preprint arXiv:1804.07230}, 2018.

\bibitem[Bourni et~al.(2019)Bourni, Langford, and Tinaglia]{bourni2019convex}
Theodora Bourni, Mat Langford, and Giuseppe Tinaglia.
\newblock Convex ancient solutions to curve shortening flow.
\newblock \emph{arXiv preprint arXiv:1903.02022}, 2019.

\bibitem[Bourni et~al.(2020)Bourni, Clutterbuck, Nguyen, Stancu, Wei, and
  Wheeler]{bourni2020ancient}
Theodora Bourni, Julie Clutterbuck, Xuan~Hien Nguyen, Alina Stancu, Guofang
  Wei, and Valentina-Mira Wheeler.
\newblock Ancient solutions for flow by powers of the curvature in
  $\mathbb{R}^2$.
\newblock \emph{arXiv preprint arXiv:2005.07642}, 2020.

\bibitem[Brendle(2018)]{brendle2018ancient}
Simon Brendle.
\newblock Ancient solutions to the ricci flow in dimension 3.
\newblock \emph{arXiv preprint arXiv:1811.02559}, 2018.

\bibitem[Brendle and Choi(2018)]{brendle2018uniqueness}
Simon Brendle and Kyeongsu Choi.
\newblock Uniqueness of convex ancient solutions to mean curvature flow in
  higher dimensions.
\newblock \emph{arXiv preprint arXiv:1804.00018}, 2018.

\bibitem[Brendle and Choi(2019)]{brendle2019uniqueness}
Simon Brendle and Kyeongsu Choi.
\newblock Uniqueness of convex ancient solutions to mean curvature flow in
  $\mathbb{R}^3$.
\newblock \emph{Inventiones mathematicae}, 217\penalty0 (1):\penalty0 35--76,
  2019.

\bibitem[Brendle et~al.(2020)Brendle, Daskalopulos, and
  Sesum]{brendle2020uniqueness}
Simon Brendle, Panagiota Daskalopulos, and Natasa Sesum.
\newblock Uniqueness of compact ancient solutions to three-dimensional ricci
  flow.
\newblock \emph{arXiv preprint arXiv:2002.12240}, 2020.

\bibitem[Caffarelli et~al.(1984)Caffarelli, Hardt, and Simon]{caffarelli1984}
Luis Caffarelli, Robert Hardt, and Leon Simon.
\newblock Minimal surfaces with isolated singularities.
\newblock \emph{Manuscripta Mathematica}, 48\penalty0 (1):\penalty0 1432--1785,
  1984.

\bibitem[Chen(2015)]{chen2015classifying}
Shibing Chen.
\newblock Classifying convex compact ancient solutions to the affine curve
  shortening flow.
\newblock \emph{The Journal of Geometric Analysis}, 25\penalty0 (2):\penalty0
  1075--1079, 2015.

\bibitem[Chodosh et~al.(2020)Chodosh, Choi, Mantoulidis, and
  Schulze]{chodosh2020mean}
Otis Chodosh, Kyeongsu Choi, Christos Mantoulidis, and Felix Schulze.
\newblock Mean curvature flow with generic initial data.
\newblock \emph{arXiv preprint arXiv:2003.14344}, 2020.

\bibitem[Choi et~al.(2020)Choi, Choi, and
  Daskalopoulos]{choi2020uniqueness-choi}
Beomjun Choi, Kyeongsu Choi, and Panagiota Daskalopoulos.
\newblock Uniqueness of ancient solutions to gauss curvature flow asymptotic to
  a cylinder.
\newblock \emph{arXiv preprint arXiv:2004.11754}, 2020.

\bibitem[Choi and Mantoulidis(2019)]{choi2019ancient}
Kyeongsu Choi and Christos Mantoulidis.
\newblock Ancient gradient flows of elliptic functionals.
\newblock \emph{arXiv preprint arXiv:1902.07697}, 2019.

\bibitem[Choi et~al.(2018)Choi, Haslhofer, and Hershkovits]{choi2018ancient}
Kyeongsu Choi, Robert Haslhofer, and Or~Hershkovits.
\newblock Ancient low entropy flows, mean convex neighborhoods, and uniqueness.
\newblock \emph{arXiv preprint arXiv:1810.08467}, 2018.

\bibitem[Choi et~al.(2019)Choi, Haslhofer, Hershkovits, and
  White]{choi2019ancient-white}
Kyeongsu Choi, Robert Haslhofer, Or~Hershkovits, and Brian White.
\newblock Ancient asymptotically cylindrical flows and applications.
\newblock \emph{arXiv preprint arXiv:1910.00639}, 2019.

\bibitem[Courant and Hilbert(1953)]{courant1953methods}
R.~Courant and D.~Hilbert.
\newblock \emph{Methods of Mathematical Physics}.
\newblock Number v. 1. Wiley, 1953.

\bibitem[Daskalopoulos et~al.(2010)Daskalopoulos, Hamilton, and
  Sesum]{daskalopoulos2010}
Panagiota Daskalopoulos, Richard Hamilton, and Natasa Sesum.
\newblock Classification of compact ancient solutions to the curve shortening
  flow.
\newblock \emph{J. Differential Geom.}, 84\penalty0 (3):\penalty0 455--464, 03
  2010.

\bibitem[Ivaki(2016)]{ivaki2016classification}
Mohammad~N Ivaki.
\newblock Classification of compact convex ancient solutions of the planar
  affine normal flow.
\newblock \emph{The Journal of Geometric Analysis}, 26\penalty0 (1):\penalty0
  663--671, 2016.

\bibitem[Loftin and Tsui(2008)]{loftin2008ancient}
John Loftin and Mao-Pei Tsui.
\newblock Ancient solutions of the affine normal flow.
\newblock \emph{Journal of Differential Geometry}, 78\penalty0 (1):\penalty0
  113--162, 2008.

\bibitem[Lunardi(2012)]{lunardi2012analytic}
Alessandra Lunardi.
\newblock \emph{Analytic semigroups and optimal regularity in parabolic
  problems}.
\newblock Springer Science \& Business Media, 2012.

\bibitem[Wang(2011)]{Wang11}
Xu-Jia Wang.
\newblock Convex solutions to the mean curvature flow.
\newblock \emph{Annals of mathematics}, pages 1185--1239, 2011.

\end{thebibliography}

\end{document}